\documentclass[11pt]{amsart}

\usepackage[margin=3.5cm]{geometry}
\usepackage{amssymb, amsmath}
\usepackage{amscd}
\usepackage[alwaysadjust]{paralist}
\usepackage{hyperref}
\usepackage[capitalize]{cleveref}
\usepackage{enumitem}
\usepackage{esint}
\usepackage{graphicx} 
\usepackage[rgb]{xcolor} 
\usepackage{verbatim}

\numberwithin{equation}{section}
\newtheorem{claim}[equation]{Claim}
\newtheorem{cor}[equation]{Corollary}
\newtheorem{lemma}[equation]{Lemma}

\newtheorem{theorem}[equation]{Theorem}
\theoremstyle{definition}

\newtheorem{rem}[equation]{Remark}
\newtheorem{rems}[equation]{Remarks}

\def\IN{\mathbb N}
\def\IR{\mathbb R}

\def\eps{\varepsilon}

\newcommand{\mult}{\operatorname{mult}}

\newcommand{\supp}{\operatorname{supp}}

\newcommand{\area}{\operatorname{area}}

\newcommand{\Sph}{\mathbb{S}}

\begin{document}

\title[Sharp asymptotics of the first eigenvalue]{Sharp asymptotics of the first eigenvalue on some degenerating surfaces}
\author{Henrik Matthiesen}
\address
{HM: Department of Mathematics, University of Chicago,
5734 S. University Ave, Chicago, Illinois 60637}
\email{hmatthiesen@math.uchicago.edu}
\author{Anna Siffert}
\address
{AS: Max Planck Institute for Mathematics,
Vivatsgasse 7, 53111 Bonn}
\email{siffert@mpim-bonn.mpg.de}
\date{\today}

\begin{abstract}
We study sharp asymptotics of the first eigenvalue on Riemannian surfaces obtained from a fixed Riemannian surface by attaching a collapsing
flat handle or cross cap to it.
Through a careful choice of parameters this construction can be used to strictly increase the first eigenvalue normalized by area if the initial surface has some symmetries.
If these symmetries are not present we show that the first eigenvalue normalized by area strictly decreases for the same range of parameters.
These results are the main motivation for the construction in \cite{MS3},
where we show a monotonicity result for the normalized first eigenvalue without any symmetry assumptions.
\end{abstract}

\maketitle

\section{Introduction}
For a closed Riemannian surface $(\Sigma,g)$ the (positive) Laplace operator acting on functions has discrete spectrum. 
We list its eigenvalues -- counted with multiplicity -- as
$$
0=\lambda_0(\Sigma,g) < \lambda_1(\Sigma,g) \leq \lambda_2(\Sigma,g) \dots \to \infty.
$$
The goal of this article is to understand the asymptotics of the scale invariant quantity 
$\lambda_1(\Sigma,g)\area(\Sigma,g)$ for a family of surfaces
$\Sigma_{\eps,h}$ obtained from the surface $\Sigma$ by attaching a flat handle $C_{\eps,h}$ or cross cap $M_{\eps,h}$ of height $h$ and radius $\eps$ that decreases -- see \cref{attach-crosscap} and \cref{attach-cylinder} below for the explicit constructions.

Variants of this problem have been studied before by various authors, see \cite{anne,anne_2,anne_3, post-diss, post, ces}, but with much less precise asymptotics than we obtain here.

\smallskip

The motivation to study this question stems from the maximization problem for the first eigenvalue normalized by area on a closed
surface of fixed topological type -- see \cite{MS1, MS3} and references therein for a short introduction to this problem.
In \cite{petrides}, Petrides proved that one can find a maximizing metric provided the sharp constant strictly increases in terms of the topology of
the surface.
A special case of this was already present in Nadirashvili's solution of Berger's isoperimetric problem \cite{nadirashvili}.
More generally, one can ask the following question:
\\

\emph{
Given a closed surface $\Sigma$, let $\Sigma'$ be obtained from $\Sigma$ by attaching a handle or a cross cap.
Can one find a metric $g'$ on $\Sigma'$ such that
$$
\lambda_1(\Sigma,g)\area(\Sigma,g) < \lambda_1(\Sigma',g')\area(\Sigma',g')\, \text{?}
$$
}

The obvious strategy that one would like to implement in order to prove such a result is to consider a family of surfaces $\Sigma_\eps$ (e.g.\ $\Sigma_{\eps,h}$ as described above) that is obtained from $\Sigma$ by attaching a tiny handle or cross cap
and 
study the asymptotics of the first eigenvalue as the handle or cross cap, respectively, collapses.
The hope is that the potential loss in the eigenvalue is compensated by the gain in area from the attached region.
It turns out that in some cases, one can in fact achieve this by means of the surfaces $\Sigma_{\eps,h}$ mentioned above.
In many other cases this seems much harder as we show that for the very same construction $\lambda_1(\Sigma,g)\area(\Sigma,g)$ strictly decreases for exactly those parameters for which the construction works under some symmetry assumption.
Still, in this case we can identify the mechanism behind this to some extent as explained in some more detail in \cref{subsec_ideas}.
This understanding is the starting point in \cite{MS3} where we give a positive answer to the above question without any restrictions on $\Sigma$ by means of a much more involved construction.

\smallskip

Before we can precisely state our main result, we need to introduce the two parameter family $\Sigma_{\eps,h}$ of surfaces that we are working with.

\subsection{Attaching a flat cross cap}\label{attach-crosscap}

Let $(\Sigma,g)$ be a closed Riemannian surface.
We fix some point $x_0 \in \Sigma$ and denote by $U$ a coordinate neighborhood containing $x_0,$
such that $g$ is conformal to the Euclidean metric in $U,$ that is $g=f g_e$ with $f$ a smooth, positive function and $g_e$ the Euclidean metric.
By dilations we may and will assume from here on that we have $f(x_0)=1$.
Let $B_{\eps^k}=B_{g_e}(x_0,\eps^k)$ be a ball centered at $x_0$ with radius $\eps^k$ with respect to $g_e$, where $k\in\mathbb{N}$.
We want to glue a cross cap 
\[
M_{\eps,h}=\Sph^1(\eps) \times [0,h]/\sim,
\]
where $(\theta,t) \sim (\theta+\pi,h-t),$
endowed with its canonical flat metric 
along its boundary to $\Sigma \setminus B_{\eps^k}$.
More precisely, we consider the surface
\[\
\Sigma_{\eps,h}:=(\Sigma \setminus B_{\eps^k}) \cup_{\partial B_\eps} M_{\eps,h},
\]
which we endow with the (non-continuous) metric $g_{\eps,h}$ given by $g$ on $\Sigma \setminus B_{\eps^k}$ and by the flat metric 
on $M_{\eps,h}.$
Below we assume $k>4$ for technical reasons.

\begin{figure}[ht]
	\centering
  \includegraphics[width=12cm]{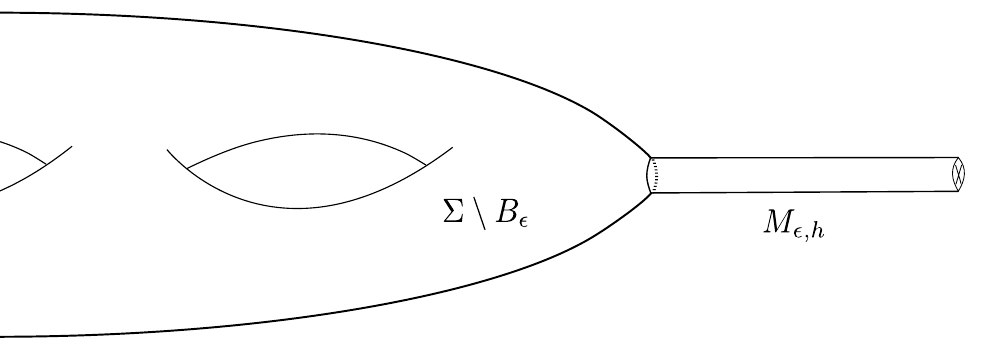}
	\caption{A part of the surface $\Sigma_{\eps,h}$}
	\label{fig2}
\end{figure}

\subsection{Attaching a flat cylinder}\label{attach-cylinder}

Similarly as above, we consider
\begin{equation*}
C_{\eps,h}=\Sph^1(\eps)\times[0,h]
\end{equation*}
endowed with its canonical flat metric.

For two distinct points $x_0, x_1 \in\Sigma$ such that $g$ is smooth near $x_i$, we take  conformally flat neighborhoods as above, which we endow with Euclidean coordinates.
We then consider the surface
\begin{equation*}
\Sigma_{\eps,h}=(\Sigma \setminus (B_{\eps^k} (x_0) \cup B_{\eps^k}(x_1) )) \cup_{\partial B_{\eps^k}(x_0) \cup \partial B_{\eps^k} (x_1)} C_{\eps,h},
\end{equation*}
where the balls $B_{\eps^k}(x_0)$ and $B_{\eps^k}(x_1)$ are again with respect to the Euclidean metric. Moreover, without loss of generality, these balls are assumed to be disjoint.
For technical reasons we again assume $k>4$.

\subsection{Main results}
Given the construction of the surfaces $\Sigma_{\eps,h}$, we can now state our first main result concerning surfaces whose first eigenfunctions all have some symmetry.

In both constructions of $\Sigma_{\eps,h}$, we restrict to parameters $h \in [h_0,h_1]$,
where $h_i$ are chosen such that
\begin{equation} \label{eq_choice_h}
\pi^2/h_1^2 < \lambda_1(\Sigma) < \pi^2/h_0^2 < \lambda_{K+1}(\Sigma),
\end{equation}
where $K=\mult(\lambda_1(\Sigma))$ denotes the multiplicity of $\lambda_1(\Sigma)$.
Note that\footnote{The notational convenience originating here is the reason for the not very geometric convention in the notation of $M_{\eps,h}$.} 
$$
\lambda_0(M_{\eps,h}) = \lambda_0(C_{\eps,h}) = \frac{\pi^2}{h^2},
$$
where $\lambda_0$ denotes the smallest Dirichlet eigenvalue.

\begin{theorem} \label{main_1}
Let $(\Sigma, g)$ be a closed Riemannian surface.
\begin{itemize}
\item[(i)]
Assume there is $x_0 \in \Sigma$ such that $\phi(x_0)=0$ for any $\lambda_1(\Sigma)$-eigenfunction $\phi$.
Then 
$$
\lambda_1(\Sigma_{\eps,h_0}) \area(\Sigma_{\eps,h_0}) > \lambda_1(\Sigma,g)\area(\Sigma,g)
$$ 
for $\eps$ sufficiently small; where $\Sigma_{\eps,h}$ is obtained from $\Sigma$ by attaching a cross cap near $x_0$ as above.
\item[(ii)]
Assume that there are distinct points $x_0,x_1 \in \Sigma$ such that $\phi(x_0) + \phi(x_1) =0$
for any $\lambda_1(\Sigma)$-eigenfunction $\phi$.
Then, for $\eps$ sufficiently small, there is $h_\eps \in [h_0,h_1]$ such that
$$
\lambda_1(\Sigma_{\eps,h_\eps}) \area(\Sigma_{\eps,h_\eps}) > \lambda_1(\Sigma,g)\area(\Sigma,g),
$$ 
where $\Sigma_{\eps,h}$ is obtained from $\Sigma$ by attaching a flat cylinder near $x_0$ and $x_1$ as above.
\end{itemize}
\end{theorem}

\begin{rems}
\begin{inparaenum}[1)]
\item
Part \textit{(ii)} of \cref{main_1} can be generalized as follows.
If there is $a >0$ and distinct points $x_0,x_1 \in \Sigma$ such that $\phi(x_0) + a \phi(x_1) =0$
for any $\lambda_1(\Sigma)$-eigenfunction $\phi$ the same result holds for a slightly adapted construction of $\Sigma_{\eps,h}$.
In this case one has to shrink the length of the fibres of the cylinder by the factor $a$ on one half of the cylinder.
For reasons of clarity we restrict ourselves to $a=1$.
\\
\item
The conclusion from the first part holds for attaching handles as well.
There are two options to obtain this.
The first option is to keep the flat product metric on $C_{\eps,h}$ but attach it close to points $x_0$ and $x_\eps$, where $d(x_\eps,x_0) \sim \eps^l$
with $1<l<k$.
The second option is to use the construction described above with $a=a_\eps \to 0$ sufficiently fast.
\end{inparaenum}
\end{rems}

As a consequence\footnote{To be precise this needs an additional smoothing argument not carried out here.
This is only a minor problem (see \cite[Section 10]{MS3} for details).}, when combined with \cite{nayatani_shoda, petrides}, we obtain the following corollary.

\begin{cor} \label{cor_exist}
There exists a maximizing
metric, smooth away from at most finitely many conical singularities, for $\lambda_1(\Sigma,g)\area(\Sigma,g)$
on the surface $\Sigma$ of genus three.
\end{cor}

\smallskip

In \cite{MS3} we provide a construction that gives the monotonicity results from \cref{main_1} for any closed Riemannian surface without any symmetry assumptions.
In particular, we obtain the analogue of \cref{cor_exist} for closed surfaces of any topological type.
The construction in \cite{MS3} is motivated by the negative result \cref{main_2} below and is significantly more involved than the construction carried out in this article.
We think that it is worth understanding the precise asymptotics for the surfaces $\Sigma_{\eps,h}$ to get an idea of the problems
that the construction in \cite{MS3} has to overcome.

\smallskip

We denote by $h_*$ the unique positive parameter such that
$$
\lambda_0(C_{\eps,h_*}) = \lambda_0(M_{\eps,h_*}) = \lambda_1(\Sigma).
$$
The range of parameters in the second part of \cref{main_1} provided by our argument is very concrete.
In particular, we have that
\begin{equation} \label{eq_sym_good_range}
h_\eps = h_* + o(\eps^{1/2}).
\end{equation}
Our second main result below gives precise asymptotics for this range of parameters $h$ if we do not have the symmetry assumption from \cref{main_1}.
In particular it shows that the first eigenvalue normalized by area decreases in this range.

Let us start with the case of attaching a cross cap.
For dimensional reasons we may choose an orthonormal basis $(\phi_0,\dots,\phi_{K-1})$ of the $\lambda_1(\Sigma)$-eigenspace such that
$$
\phi_1(x_0) = \dots = \phi_{K-1}(x_0) =0.
$$
and
$$
\phi_0(x_0) \geq  0.
$$

Fix some large $D>0$ and let $f_\eps \colon [h_*-D\eps^{1/2},h_*+D\eps^{1/2}] \to (0,\infty)$  be the unique positive
function\footnote{Note that for $h$ fixed the equation below is a quadratic equation for $f_\eps$, that has two real solutions with different signs.}
given implicitly
by
\begin{equation} \label{eq_def_f}
f_\eps^2 -1 = - \left(\frac{h}{2\pi}\right)^{3/2} \frac{\lambda_1(\Sigma)-\lambda}{\phi_0(x_0)} \eps^{-1/2} f_\eps,
\end{equation}
where $\lambda$ denotes the Dirichlet eigenvalue $\lambda_0(M_{\eps,h})$ and we already assume that $\phi_0(x_0) > 0$.

\begin{theorem} \label{main_2}
Let $\Sigma_{\eps,h}$ be obtained by attaching a cross cap as above and assume $\phi_0(x_0) >0$.
Then we have that
\begin{equation} \label{eq_main_2}
\lambda_1(\Sigma_{\eps,h}) = \lambda_1(\Sigma) - f_\eps(h)^{-1} \lambda_1(\Sigma) \phi_0(x_0) \eps^{1/2} + O(\eps \log(1/\eps))
\end{equation}
uniformly for $\eps \to 0$ and $h \in [h_*-D\eps^{1/2},h_*+D\eps^{1/2}]$ as long as $D>0$ is fixed,
where $f_\eps$ is defined by \eqref{eq_def_f}.
\end{theorem}

\begin{rem}
There is an analogous result for the case of attaching a cylinder near points $x_0$ and $x_1$
such that there is a $\lambda_1(\Sigma)$-eigenfunction $\phi$ with
$
\phi(x_0)+\phi(x_1) \neq 0.
$
In this case we may
choose an orthonormal basis $(\phi_0,\dots,\phi_{K-1})$ of the $\lambda_1(\Sigma)$-eigenspace such that
$$
\phi_1(x_0) + \phi_1(x_1) = \dots = \phi_{K-1}(x_0) + \phi_{K-1}(x_1) =0.
$$
and
$$
\phi_0(x_0) + \phi_0(x_1) > 0.
$$
One then has a similar expansion with $f_\eps$ the unique positive function 
defined by
\begin{equation} \label{eq_def_f_cyl}
f_\eps^2-1
=
-\frac{1}{2}  \left(\frac{h}{\pi}\right)^{3/2} \frac{\lambda_1(\Sigma)-\lambda}{\phi_0(x_0)+\phi_0(x_1)} \eps^{-1/2} f_\eps
\end{equation}
\end{rem}

\begin{rem}
With some minor changes our techniques can be adapted to show that the conclusion of \cref{main_2} holds for $k=1$, as well.
This still works since $\eps^{1/2} \gg \eps \log(1/\eps)$ for $\eps$ small.
In contrast to this, we do not know if the same applies to \cref{main_1}, since $\eps \log(1/\eps) \gg \eps$ for $\eps$ small.
\end{rem}

\smallskip

In order to reduce the technicalities a bit we only provide the proof in the case in which the cross cap is attached to a point in which not all the eigenfunctions vanish.
The argument in the case of a cylinder is completely analogous but the computation are longer due to more lower order correction terms necessary in that case.

It is worth pointing out that 
$$
f_\eps(h_*)=1
$$
for any $\eps>0$.
More generally, $f_\eps$ is positive and uniformly bounded from below on scales $h_* \pm \tau \eps^{1/2}$ for fixed $\tau >0$.

The key point to prove is that the function $f_\eps$ describes the ratio of concentration on $\Sigma$ and $M_{\eps,h}$ for the first eigenfunction on $\Sigma_{\eps,h}$.
More precisely, up to a small error term, we have that 
\begin{equation} \label{eq_meaning_f}
f_\eps(h) = \frac{\|u_{\eps,h}\|_{L^2(\Sigma \setminus B_{\eps^k})}}{\|u_{\eps,h}\|_{L^2(M_{\eps,h})}},
\end{equation}
where $u_{\eps,h} \colon \Sigma \to \IR$ is a normalized $\lambda_1(\Sigma_{\eps,h})$-eigenfunction.

With some more care for the error terms our arguments can in fact be used to improve \cref{main_2}, e.g.\ to uniform control in $[h_*+\eps^{1/3},h_*-\eps^{1/3}]$.
In view of \eqref{eq_sym_good_range} the parameters on scales $h_* \pm \eps^{1/2}$ seem to be the most interesting ones.
Also the main transition happens at these scales: 
Consider the (a priori not necessarily continuous) function $c_\eps \colon h \mapsto \|u_{\eps,h}\|_{L^2(\Sigma \setminus B_{\eps^k})}$,
where $u_{\eps,h}$ is a choice of a normalized $\lambda_1(\Sigma_{\eps,h})$-eigenfunction.
For $\eps$ fixed but very small, this function is close to $1$
for $h=h_0$ and close to $0$ for $h=h_1$.
In fact, we even have the following much stronger conclusion.
Given any $r \in (0,1)$, there is $D>0$ such that $\|u_{\eps,h}\|_{L^2(\Sigma \setminus B_{\eps^k})} \in (0,r) \cup (1-r,1)$ outside of $[h_*-D\eps^{1/2},h_*+D\eps^{1/2}]$ (for $\eps$ sufficiently small depending on $r$.)
On the other hand, the change of $c_{\eps}$ from $r$ to $1-r$ is precisely described through the function $f_\eps$.
In particular, it does not come from any discontinuity of $c_\eps$ but from the first eigenfunction being simple and changing its concentration of $L^2$-norm.

\subsection{Main problems and ideas} \label{subsec_ideas}
Our analysis rests on two ingredients: Firstly, pointwise bounds on eigenfunctions with bounded energy along the boundary of the attached regions.
Secondly, optimal approximate solutions to the eigenvalue equation on $\Sigma_{\eps,h}$ constructed from solutions on $\Sigma$ and the collapsing flat part, respectively.

The pointwise bounds on eigenfunctions are a bit subtle because of the discontinuity of the metric in precisely the region we are working in.
However, this can be obtained from standard elliptic estimates by scaling and an application of the De Giorgi--Nash--Moser estimates.
This exploits the fact that the discontinuity of the metric is purely conformal in nature and that the Laplace operator is conformally covariant in dimension two.

Having these bounds at hand we can then proceed to show that the spectrum of $\Sigma_{\eps,h}$ resembles the union of the spectra of $\Sigma$ and $[0,h]$ (with Dirichlet boundary conditions) for $\eps$ sufficiently small.
A similar conclusion follows on the level of eigenfunctions.
We show that these are close in $L^2$ to a linear combination of a $\lambda_1(\Sigma)$- and a $\lambda_0(C_{\eps,h})$-eigenfunction.  
At this stage we do not have any sufficiently strong control on the rate of convergence to conclude our main results.

In order to improve our estimates on the rate of convergence we then construct approximate solutions to the eigenvalue equation on $\Sigma_{\eps,h}$ of two types.
Let us consider the case of attaching a flat cross cap $M_{\eps,h}$ first.

The first type is given by extending $\lambda_1(\Sigma)$-eigenfunctions appropriately.
More precisely, we start from $\phi \colon \Sigma \to \IR$ an $L^2(\Sigma)$-normalized $\lambda_1(\Sigma)$-eigenfunction.
We then extend this after a suitable interpolation by the constant $\phi(x_0)$  to $M_{\eps,h}$.
By testing an eigenfunction $u_{\eps,h}$ against such a quasimode one obtains an asymptotic expansion of the corresponding eigenvalue.
The largest error term in this expansion contains the term
\begin{equation} \label{eq_ideas_1}
\phi(x_0) \int_{ M_{\eps,h}} u_{\eps,h}.
\end{equation}

The case of the flat cylinders is in principal similar.
However, in the situation of the second item of \cref{main_1} there is a more sensitive way of extending $\phi$ to $C_{\eps,h}$ exploiting the fact that $\mu_1(C_{\eps,h})=\lambda_0(C_{\eps,h})$ and that the Neumann eigenfunction glues well at the two boundary components precisely because of the symmetry assumption.
This gives a much better approximate solution than the first construction at least when $h$ is close to $h_*$, i.e.\ when $|\lambda_1(\Sigma)-\lambda_0(C_{\eps,h})|$ is small.

The second type of approximate solutions is constructed out of a normalized $\lambda_0(M_{\eps,h})$-eigenfunction $\psi_{\eps,h}$.
Because of the nature of the collapse of $M_{\eps,h}$ we have that $\psi_{\eps,h}=\eps^{-1/2}\psi_{h}$, so that
\begin{equation} \label{eq_ideas_2}
\int_{\partial M_{\eps,h}}\partial_\nu \psi_{\eps,h}\, d \mathcal{H}^1 \sim \eps^{1/2},
\end{equation}
which is essentially the scale on which $\psi_{\eps,h}$ fails to solve the eigenvalue equation on $\Sigma_{\eps,h}$, cf.\ the discussion in the beginning of \cref{sec_quasi_handle}.
In turns out that this is typically sharp. 
In order to construct an optimal approximate solution we use Green's kernel of $\Delta-\lambda_0(M_{\eps,h})$
with pole at $x_0$ scaled by roughly $\eps^{1/2}$ in order to cancel out the normal derivative \eqref{eq_ideas_2}.
After taking care of the presence of the non-trivial kernel of this operator at $h=h_*$ and testing against an eigenfunction $u_{\eps,h}$ this gives an asymptotic expansion with leading order error term containing
\begin{equation} \label{eq_ideas_3}
\sum_{i=0}^{K-1} \phi_i(x_0) \int_{\Sigma \setminus B_{\eps^k}} \phi_i u_{\eps,h} \eps^{1/2},
\end{equation}
where $(\phi_0,\dots,\phi_{K-1})$ is an orthonormal basis of $\lambda_1(\Sigma)$-eigenfunctions.

Using the approximate decomposition of eigenfunctions, we can also show that the term at \eqref{eq_ideas_1} is comparable to $n \eps^{1/2}$,
if $u_{\eps,h}$ is given approximately by $n \psi_{\eps,h}$ for the normalized, positive $\lambda_0(M_{\eps,h})$-eigenfunction $\psi_{\eps,h}$.
On the other hand, a similar argument gives that \eqref{eq_ideas_3} can be related to the $L^2$-norm of $\left. u_{\eps,h} \right|_{\Sigma \setminus B_{\eps^k}}$.
In summary, we find that testing against a quasimode of one type, the leading order correction term corresponds exactly to the other type of the spectrum.
While this gives a first strong hint on the interaction of the two parts of the spectrum, our proof of \cref{main_2} is actually a bit different exploiting the correction term on the next largest scale for the quasimodes concentrated on $M_{\eps,h}$.

Finally, let us return to the second item of \cref{main_1}.
Besides the improved quasimodes concentrated on $\Sigma$ that we construct in this case also the quasimodes concentrated on $C_{\eps,h}$
have a more favourable behaviour in this situation.
We have to use a sum of two Green's functions with poles at $x_0$ and $x_1$, respectively.
Therefore, \eqref{eq_ideas_3} turns into
$$
\sum_{i=0}^{K-1} (\phi_i(x_0)+\phi_i(x_1)) \int_{\Sigma \setminus B_{\eps^k}} \phi_i u_{\eps,h} \eps^{1/2} = 0
$$
improving the convergence rate to $\eps \log(1/\eps)$, which turns out to be strong enough to conclude by choosing $h=h_\eps$ appropriately.

\subsubsection{A few words on \cite{MS3}}
In \cite{MS3} we provide a much more involved construction attaching a truncated, degenerating hyperbolic cusp to $\Sigma$ in order to obtain the monotonocity of the normalized first eigenvalue without any additional assumptions.
Ultimately, the sharp convergence rate on scale $\eps^{1/2}$ in \cref{main_2} orginates from \eqref{eq_ideas_2} and should be expected as long as the collpasing part
resembles a model with an isolated eigenvalue at $\lambda_0$ in the limit.
While the techniques from this article give the negative result \cref{main_2} they do not apply to get any convergence rate below $\eps \log(1/\eps)$.
Therefore, in \cite{MS3} we have to develop an entirely different approach.
Of course, we still rely on some of the more technical ideas from here in particular the robust pointwise bounds on eigenfunctions
 \cref{max_eigen_neu_ell_est} and the construction of optimal quasimodes in \cref{sec_quasi_handle}. 
We would also like to point out that there is some connection to the second item of \cref{main_1}.
Recall that its proof crucially relies on the fact that $\lambda_0(C_{\eps,h})=\mu_1(C_{\eps,h})$.
Maybe surprisingly, the analogous fact for the truncated hyperbolic cusp is one of the driving forces of the  proof in \cite{MS3} but exploited in a completely different way.

For the sake of readability and in order to keep both papers self-contained we decided to include the corresponding arguments here and in \cite{MS3}.
In particular, corresponding versions of the robust pointwise bound \cref{max_eigen_neu_ell_est} and the construction of good quasimodes in \cref{sec_quasi_handle} are  two of the key technical ingredients in \cite{MS3}.

\smallskip

\textbf{Outline.}
\cref{sec3} contains pointwise estimates for the eigenfunctions of $\Sigma_{\eps,h}$.
The spectrum of $\Sigma_{\eps,h}$ and the convergence of the eigenfunctions on $\Sigma_{\eps,h}$
as $\eps \to 0$ are discussed in \cref{limit}.
 In \cref{sec4} we construct approximate eigenfunctions on $\Sigma_{\eps,h}$ which will be used in \cref{sec5} to prove the main results, 
 i.e.\
 \cref{main_1} and \cref{main_2}.

\smallskip

\textbf{Acknowledgements.}
The first named author would like to thank his former advisor Werner Ballmann for a helpful discussion on Green's functions.
The second named author would like to thank the Max Planck Institute for Mathematics in Bonn for financial support and excellent working conditions.
We would also like to thank the anonymous referee for an extremely detailed report that helped us to significantly improve the presentation and readability of the manuscript.

\section{Pointwise Estimates for Eigenfunctions}\label{sec3} 
In this section we provide estimates for the eigenfunctions in the attaching region.
We will use these later to obtain closeness of the restrictions to the collapsing part to Dirichlet-eigenfunctions.

Let $x \in \Sigma$ be the center of a ball $B_{\eps^k}(x)$ which is removed from $\Sigma$ in the construction of $\Sigma_{\eps,h}$.
In the case of attaching a cylinder we have $x \in \{x_0,x_1\}$, in the case of attaching a cross cap we have $x=x_0$.

In order to understand the spectrum of $\Sigma_{\eps,h}$, we need some bounds for eigenfunctions with bounded energy on $\partial B_{\eps^k}(x)$.
For ease of notation, we assume that the ball $B_1(x) \subset \Sigma$
can be endowed with conformal coordinates.
In the case of attaching a cylinder, we also assume that the two balls $B_1(x_0)$ and $B_1(x_1)$ are disjoint.

\begin{lemma} \label{max_eigen_neu_ell_est}
Let $u_{\eps,h}$ be an $L^2$-normalized eigenfunction on $\Sigma_{\eps,h}$ with eigenvalue $\lambda_{\eps,h} \leq \Lambda$.
There is a constant $C$ depending on $\Lambda$ and $k$ (from the construction of $\Sigma_{\eps,h}$), such that the following holds.
If we use Euclidean polar coordinates $(r,\theta)$ centered at $x$ we have the uniform pointwise bounds
\begin{equation} \label{max_eigen_sup_bd}
	|u_{\eps,h}|(r,\theta) \leq C \log \left(1/r \right),
\end{equation}
for $ \eps^k \leq r \leq 1/2$
and
\begin{equation} \label{max_eigen_lipschitz_bd}
	|\nabla u_{\eps,h}| (r, \theta) \leq C/r
\end{equation}
for $2 \eps^k \leq r \leq 1/2.$
\end{lemma}

Note that \cref{max_eigen_neu_ell_est} is related\footnote{It is not hard to improve \eqref{max_eigen_sup_bd} to $\log^{1/2}(1/\eps)$, but we do not need this.}
 to the integral bound
\begin{equation} \label{eq_int_bd_sobolev}
\int_{\partial B_{\eps^k}(x)} |\varphi|^2 d \mathcal{H}^1
\leq 
C \eps^k \log(1/\eps^k) \|\varphi \|_{W^{1,2}(B_1(x) \setminus B_{\eps^k} (x))}^2
\end{equation}
that holds for any $\varphi \in W^{1,2}(\Sigma_{\eps,h})$ and which can be proved by a straightforward computation in polar coordinates.

\begin{proof}
Recall that we have identified a conformally flat neighborhood of $x$ with $B_1=B(0,1)\subset \IR^2,$
such that $x=0.$
First, observe that, up to radius $2\eps^k$ \eqref{max_eigen_sup_bd} is a direct consequence of \eqref{max_eigen_lipschitz_bd}.
In fact, by the standard elliptic estimates \cite[Chapter 5.1]{taylor_1},  the functions $u_{\eps,h}$ are 
uniformly bounded in $C^\infty$
within compact subsets of $\Sigma \setminus \{x_0\}$.
Given this, we can integrate the bound \eqref{max_eigen_lipschitz_bd} from $\partial B_{1/2}$ to $\partial B_r$ and find \eqref{max_eigen_sup_bd}.

The bound \eqref{max_eigen_lipschitz_bd} follows from standard elliptic estimates after rescaling the scale $r$ to a fixed scale.
More precisely, we consider the rescaled functions $w_r(z):=u_{\eps,h}(r z).$
On $B_1\setminus B_{\eps^k}$ the metric of $\Sigma$ is uniformly bounded from above and below by the Euclidean metric.
Hence we can perform all computations in the Euclidean metric.

Since the Laplace operator is conformally covariant in dimension two, $w_r$ solves the equation
\begin{equation} \label{max_eigen_eqn_sc}
\Delta_e w_r=r^2 f_r \lambda_{\eps,h} w_r,
\end{equation}
with $f_r(z)=f(rz)$ a smooth function and $\Delta_e$ the Euclidean Laplacian.
Since $f \in C^\infty,$ we have uniform $C^\infty$-bounds on $f_r$ for $r \leq 1.$
Taking derivatives, we find that
\begin{equation} \label{max_eigen_eqn_der}
\Delta_e \nabla w_r = r^2 \lambda_{\eps,h} \nabla (f_r w_r),
\end{equation}
where also the gradient is taken with respect to the Euclidean metric.
Since $\lambda_{\eps,h} \leq \Lambda$ the scaling invariance of the Dirichlet energy implies that
\begin{equation*}
\begin{split}
\lambda_\eps^2 \int_{B_3 \setminus B_{1/2}}  |\nabla (f_r w_r)|^2 
&=
 \lambda_\eps^2 \int_{B_{3r} \setminus B_{r/2}} |\nabla (fu_{\eps,h})|^2
 \\
 &\leq
 2 \lambda_{\eps,h}^2 \int_{B_{3r} \setminus B_{r/2}} f^2 |\nabla u_{\eps,h}|^2 + u_{\eps,h}^2 |\nabla f|^2
 \\
&\leq
 C
\end{split}
\end{equation*}
by assumption.
In particular,
the right hand side of \eqref{max_eigen_eqn_der} is bounded by 
$Cr^2$ in $L^2(B_3\setminus B_{1/2}).$
Therefore, by standard elliptic estimates \cite[Chapter 5.1]{taylor_1}, we have
\begin{equation*}
\sup_{ \{1\leq s \leq 2\} }|\nabla w_r|(s, \theta) \leq C r^2 + C |\nabla w_r|_{L^2(B_3 \setminus B_{1/2})} \leq C,
\end{equation*}
which scales to
\begin{equation*}
\sup_{\{r\leq s \leq 2r\}}|\nabla u_{\eps,h}| (s,\theta) \leq C/r,
\end{equation*}
with $C$ independent of $r.$
This proves the estimate \eqref{max_eigen_lipschitz_bd}, hence also \eqref{max_eigen_sup_bd} for $r \geq 2 \eps^k$ as explained above.

To get the estimate \eqref{max_eigen_sup_bd} for the remaining radii we invoke the De Giorgi--Nash--Moser estimate.
For $\alpha>1$ consider the two sets
$$
U_{\eps,h}^{(1)}(\alpha)= B_{\alpha \eps^k} \setminus B_{\eps^k} \subset \Sigma 
$$
and
$$
U_{\eps,h}^{(2)}(\alpha) = \mathbb{S}^1(\eps) \times [0,\alpha \eps) \subset M_{\eps,h} \ \text{or} \ C_{\eps,h}.
$$
Then
\begin{equation*}
U_{\eps,h}(\alpha)=
U_{\eps,h}^{(1)} \cup U_{\eps,h}^{(2)}
\end{equation*}
is a neighbourhood of $\partial B_{\eps^k}$ in $\Sigma_{\eps,h}$ which comes with canonical (singularly) conformal coordinates
$$
\Phi_{\eps,h} \colon V(4)  \to U_{\eps,h}(4),
$$
where we write  
$$
V(\alpha)=(B_\alpha \setminus B_1) \cup_{\partial B_1} (\mathbb{S}^1 \times [0,\alpha)).
$$
Moreover, we write
\begin{equation*}
f_{\eps,h}
= 
\begin{cases}
\eps^{-2k} & \ \text{in} \ B_4 \setminus B_1
\\
\eps^{-2} & \ \text{in} \ \mathbb{S}^1 \times [0,4)
\end{cases}
\end{equation*}
Note that the metric 
$$
l_{\eps,h}=f_{\eps,h}\Phi_{\eps,h}^* g_{\eps,h}
$$ 
is uniformly bounded from above and below almost everywhere by a fixed metric.
In fact, on $S^1 \times [0,4)$ the metric $l_{\eps,h}$ is the metric of a fixed flat cylinder, and on $B_4 \setminus B_1$ the metric $l_{\eps,h}$ is close to the standard flat metric on (a subset of) the unit disk.
Consider the function defined on $V(4)$ by
$$
w_{\eps,h} = u_{\eps,h} \circ \Phi_{\eps,h} - (u_{\eps,h} \circ \Phi_{\eps,h})_{V(4)} 
$$
where $(\cdot)_{V(4)}$ denotes the mean value of a function on $V(4)$ with respect to the metric $l_{\eps,h}$.
By the conformal invariance of the Dirichlet energy, we find that $w_{\eps,h}$ has gradient bounded in $L^2$ with respect to the rescaled metric,
\begin{equation} \label{dnm_1}
\int_{V(4)} |\nabla w_{\eps,h}|^2 dA_{l_{\eps,h}}\leq C.
\end{equation}
Since $l_{\eps,h}$ is uniformly controlled from above and below, there is a constant $C$ independent of $\eps,h$ 
such that
\begin{equation} \label{dnm_2}
\int_{V(4)} |w_{\eps,h}|^2 dA_{l_{\eps,h}} \leq C \int_{V(4)} |\nabla w_{\eps,h}|^2 dA_{l_{\eps,h}}.
\end{equation}
Next observe that $w_{\eps,h}$ is a weak solution to the equation 
\begin{equation} \label{conf_diff_eqn}
\Delta_{l_{\eps,h}} w_{\eps,h} = \frac{1}{f_{\eps,h}} \Delta_{g_{\eps,h}} u_{\eps,h} = \frac{1}{f_{\eps,h}} \lambda_{\eps,h} u_{\eps,h},
\end{equation}
thanks to the conformal covariance of the Laplacian in dimension two, which is easily checked to hold also in the singular context required for the above equation.
Finally, note that the right hand side of \eqref{conf_diff_eqn} is bounded in $L^2(V(4),dA_{l_{\eps,h}})$.
Thanks to this, \eqref{dnm_1}, \eqref{dnm_2}, and \eqref{conf_diff_eqn} we can apply the inhomogeneous De Giorgi--Nash--Moser estimates (see e.g.\ \cite[Theorem 8.17]{gilbarg_trudinger}) to obtain
\begin{equation*}
\sup_{p,q\in V(2)} |w_{\eps,h}(p)-w_{\eps,h}(q)| \leq C.
\end{equation*}
Since this is scale invariant, independent of $\eps,h$ and this implies
 \eqref{max_eigen_sup_bd}.
\end{proof}

We have a similar, but much less subtle bound in the case of Neumann eigenfunctions.

\begin{lemma} \label{neu_ell_est}
Let $u_\eps$ be a normalized $\mu_1(\Sigma \setminus B_{\eps^k}(x))$-eigenfunction.
If we use Euclidean polar coordinates $(r,\theta)$ centered at $x,$ we have the uniform pointwise bound
\begin{equation} \label{sup_bd}
	|u_{\eps}|(r,\theta) \leq C \log \left(\frac{1}{r} \right),
\end{equation}
for any $\eps^k \leq r \leq 1/2.$
\end{lemma}

\begin{proof}
For radii $r \geq 2 \eps^k$ this follows from the proof of \cref{max_eigen_neu_ell_est}.
For the remaining radii, we use the same argument but apply elliptic boundary estimates
\cite[Chapter 5.7]{taylor_1}.
\end{proof}

\section{The limit spectrum}\label{limit}
In  this section we discuss the spectrum of $\Sigma_{\eps,h}$ and the convergence of the eigenfunctions on $\Sigma_{\eps,h}$
as $\eps \to 0$. We mainly restrict our discussion to the surfaces 
$\Sigma_{\eps,h}=(\Sigma \setminus B_\eps)\cup_{\partial B_\eps} M_{\eps,h}.$
The discussion for glueing handles is similar or identical. 
We will indicate the necessary changes.

For fixed $h >0$ denote by 
$$0=\nu_{0,h} < \nu_{1,h} \leq \nu_{2,h} \leq \dots$$ 
the reordered union (counted with multiplicity) of the eigenvalues of $\Sigma$ and of those Dirichlet eigenvalues on $M_{\eps,h}$ that correspond to rotationally symmetric functions.
Note that the latter are precisely the limits of eigenvalues on $M_{\eps,h}$ has $\eps \to 0$.

Also, for $u \in W^{1,2}(\Sigma \setminus B_\eps),$ we write $\tilde u \in W^{1,2}(\Sigma)$
for the function which is given by $u$ in $\Sigma \setminus B_\eps$
and by the harmonic extension of $\left. u \right|_{\partial B_\eps}$
to $B_\eps.$

\begin{theorem} \label{conv}
For any $l \in \IN$ we have that
$$
\lim_{\eps \to 0} \lambda_l(\Sigma_{\eps,h}) = \nu_{l,h}
$$
uniformly in $h \in [h_0,h_1]$.
Moreover, for a sequence of normalized eigenfunctions $u_{\eps,h}$ on $\Sigma_{\eps,h}$ with uniformly bounded eigenvalue we have subsequential convergence
as follows.
\begin{enumerate}
\item
On $\Sigma$ we have that
$$
\widetilde{\left. u_{\eps}\right|_{\Sigma \setminus B_{\eps}}} \to \phi
$$ in
$L^2(\Sigma),$ 
where 
$\phi$ is an eigenfunction on $\Sigma$;
and
\item 
On $M_{\eps,h}$ we have that
$$
\int_{M_{\eps,h}} \left| u_{\eps,h}- \left(\int_{M_{\eps,h}} u_{\eps,h} \eps^{-1/2} \psi_h \right) \eps^{-1/2} \psi_h \right|^2 \leq C \eps \log(1/\eps),
$$
where
 $\eps^{-1/2} \psi_h$ is a normalized rotationally symmetric Dirichlet eigenfunction on $M_{\eps,h}$.
\end{enumerate}
\end{theorem}

Most of this material is contained in \cite{anne_2, post-diss,post}, where the case of handles of fixed height $h$ and $k=1$ is covered.
The key ingredient for the case $k>1$ is the pointwise bound from \cref{max_eigen_neu_ell_est}.
The quantitative estimate in the second item above seems to be new.
It is a crucial ingredient to obtain \cref{main_2}.

For the proof of \cref{conv} we need the following result for the
Neumann spectrum of $\Sigma \setminus B_{\eps^k}$, which can also be found in \cite{anne}.

\begin{lemma} \label{max_eigen_neumann_conv}
The spectrum of $\Sigma \setminus B_{\eps^k}$ with Neumann boundary conditions
converges to the spectrum of $\Sigma.$
Moreover, for any sequence $\eps_l \to 0$ and orthonormal eigenfunctions
$u_1^{\eps_l},\dots u_k^{\eps_l}$ on $\Sigma \setminus B_{\eps_l},$
with uniformly bounded eigenvalues, 
we have subsequential convergence $\tilde u_i^{\eps_l} \to u_i$ in $L^2(\Sigma),$
where $u_1,\dots,u_k$ are orthonormal eigenfunctions on $\Sigma.$
\end{lemma}

Since some steps in the proof are very similar to the argument for \cref{conv} we defer the proof for a moment.

\begin{proof}[Proof of \cref{conv}]

\textsc{Step 1:} 
\textit{Asymptotic upper bound}
\smallskip

Let $\eta_\eps$ be a log cut-off function,
$$
\eta_\eps =
\begin{cases}
1 & \text{in} \ \Sigma \setminus B_{\eps^{k/2}}(x_0)
\\
1-\frac{\log(|x|/\eps^{k/2})}{\log(\eps^{k/2})} & \text{in} \ B_{\eps^{k/2}}(x_0)
\\
0 & \ \text{else},
\end{cases}
$$
and
$\phi \colon \Sigma \to \IR$ be a normalized eigenfunction with eigenvalue $\lambda$, then
\begin{equation} \label{eq_upper_bd}
\begin{split}
\int_{\Sigma_{\eps,h}} |\nabla(\eta_{\eps} \phi)|^2
&\leq 
\int_{\Sigma_{\eps,h}} \eta_\eps^2 |\nabla \phi|^2 
+
2 \int_{\Sigma_{\eps,h}}  \nabla \eta_\eps \cdot \nabla \phi
+
\int_{\Sigma_{\eps,h}} \phi^2 |\nabla \eta_\eps|^2
\\
& \leq 
\int_{\supp \eta_\eps} |\nabla \phi|^2
+ 2 \sup_\Sigma |\nabla \phi| \left(\int_{\Sigma} |\nabla \eta_\eps|^2 \right)^{1/2}
+ \sup_\Sigma |\phi| \int_\Sigma |\nabla \eta_\eps|^2
\\
& \leq \lambda \int_\Sigma \phi^2 +
\frac{C}{\log^{1/2}(1/\eps)}
+
\frac{C}{\log(1/\eps)}
\\
& \leq
\lambda
+
\frac{C}{\log^{1/2}(1/\eps)}
\end{split}
\end{equation}
since $\phi$ and $|\nabla \phi|$ are bounded and by the explicit choice of $\eta_\eps$.
Similarly, for two orthogonal eigenfunctions $\phi_1,\phi_2$ we have that
$$
\left| \int_{\Sigma_{\eps,h}} (\eta_\eps \phi_i)(\eta_\eps \phi_j) - \delta_{ij} \right|
+ \left| \int_{\Sigma_{\eps,h}} \nabla(\eta_\eps \phi_i) \nabla(\eta_\eps \phi_j) - \delta_{ij} \right|
\leq C \eps^{2k} + \frac{C}{\log^{1/2}(1/\eps)}.
$$ 
Moreover, for any two Dirichlet eigenfunctions on $M_{\eps,h}$ their extension by $0$ to all of $\Sigma_{\eps,h}$ are clearly orthogonal in $L^2(\Sigma)$ and $W^{1,2}(\Sigma)$ and have disjoint support with all the functions $\eta_\eps \phi$ as above.
The asymptotic upper bound on the eigenvalues follows now immediately from the variational characterization of the eigenvalues.
\\

\textsc{Step 2:} 
\textit{Asymptotic lower bound}
\smallskip

For $u_{\eps,h}$ an eigenfunction on $\Sigma_{\eps,h}$ we denote by $v_{\eps,h}$ the harmonic extension of $\left. u_{\eps,h} \right|_{\partial M_{\eps,h}}$
to $M_{\eps,h}$.
If then $u_{\eps,h}$ is a normalized eigenfunction with uniformly bounded eigenvalue on $\Sigma_{\eps,h}$, it follows from the maximum principle and \cref{max_eigen_neu_ell_est}, that
\begin{equation*}
\sup_{M_{\eps,h}} |v_{\eps,h}|
\leq \sup_{\partial M_{\eps,h}} |v_{\eps,h}|
=
\sup_{\partial B(x,\eps^k)} |u_{\eps,h}|
 \leq 
C \log(1/\eps^k).
\end{equation*}
This implies
\begin{equation} \label{eq_est_harm}
\int_{M_{\eps,h}} |v_{\eps,h}|^2 \leq C \eps  |\log(\eps)|.
\end{equation}
uniformly in $h \in [h_0,h_1]$.

Let now $w_{\eps,h}$ be a normalized linear combination of the first $(l+1)$-eigenfunctions on $\Sigma_{\eps,h}$ and write $t_{\eps,h}$
for the harmonic extension of
$\left. w_{\eps,h} \right|_{\partial M_{\eps,h}}$ to $M_{\eps,h}$.
For dimensional reasons, we may choose $w_{\eps,h}$ 
orthogonal to the first $m$ Neumann eigenfunctions on $\Sigma \setminus B_{\eps^k}$ and such that
$w_{\eps,h} - t_{\eps,h}$ is orthogonal
to the first $n$ Dirichlet eigenfunctions on $M_{\eps,h}$ provided $m+n \leq l$.

First note that since $\left. w_{\eps,h} \right|_{M_{\eps,h}} - t_{\eps,h} \in W_0^{1,2}(M_{\eps,h})$,
we obtain from integration by parts that
\begin{equation*}
\int_{M_{\eps,h}} \nabla (w_{\eps,h}  - t_{\eps,h}) \cdot \nabla t_{\eps,h}=0.
\end{equation*}
This is turn implies that
\begin{equation*}
\int_{M_{\eps,h}} | \nabla (w_{\eps,h}  - t_{\eps,h}) |^2 = \int_{M_{\eps,h}} | \nabla w_{\eps,h}|^2 - \int_{M_{\eps,h}} |\nabla t_{\eps,h}|^2 
\leq  \int_{M_{\eps,h}} |\nabla w_{\eps,h}|^2.
\end{equation*}

We then find that
\begin{equation*}
\begin{split}
\int_{\Sigma_{\eps,h}} |\nabla w_{\eps,h}|^2
&\geq 
\int_{\Sigma \setminus B_{\eps^k}} |\nabla w_{\eps,h}|^2
+
\int_{M_{\eps,h}} |\nabla (w_{\eps,h}-t_{\eps,h})|^2
\\
& \geq
\mu_{m}(\Sigma \setminus B_{\eps^k}) \int_{\Sigma \setminus B_{\eps^k}} |w_{\eps,h}|^2
+
\lambda_{n}(M_{\eps,h}) \int_{M_{\eps,h}} |w_{\eps,h} - t_{\eps,h}|^2
\\
& \geq
\mu_{m}(\Sigma \setminus B_{\eps^k}) \int_{\Sigma \setminus B_{\eps^k}} |w_{\eps,h}|^2
+
\lambda_{n}(M_{\eps,h}) \int_{M_{\eps,h}} |w_{\eps,h} |^2
\\
& \ \ -
2 \lambda_n(M_{\eps,h}) \int_{M_{\eps,h}} w_{\eps,h} t_{\eps,h}
+
\lambda_n(M_{\eps,h}) \int_{M_{\eps,h}} |t_{\eps,h}|^2
\\
& \geq
\mu_{m}(\Sigma \setminus B_{\eps^k}) \int_{\Sigma \setminus B_{\eps^k}} |w_{\eps,h}|^2
+
\lambda_{n}(M_{\eps,h}) \int_{M_{\eps,h}} |w_{\eps,h} |^2
\\
& \ \ -
2 \lambda_n(M_{\eps,h}) \left( \int_{M_{\eps,h}} |w_{\eps,h}|^2 \right)^{1/2} \left( \int_{M_{\eps,h}} |t_{\eps,h}|^2 \right)^{1/2}
\\
& \geq
\mu_{m}(\Sigma \setminus B_{\eps^k}) \int_{\Sigma \setminus B_{\eps^k}} |w_{\eps,h}|^2
+
\lambda_{n}(M_{\eps,h}) \int_{M_{\eps,h}} |w_{\eps,h}|^2 - C \eps^{1/2} \log^{1/2}(1/\eps^k)
\\
& \geq
\min\{\mu_{m}(\Sigma \setminus B_{\eps^k}) , \lambda_{n}(M_{\eps,h})\} \int_{\Sigma_{\eps,h}} |w_{\eps,h}|^2 - C \eps^{1/2} \log^{1/2}(1/\eps^k),
\end{split}
\end{equation*}
where we have used \eqref{eq_est_harm} and our choice of $w_{\eps,h}$.
The asymptotic lower bound now follows easily by choosing $m$ and $n$ appropriately using \cref{max_eigen_neumann_conv}.
\\

\textsc{Step 3:} 
\textit{Convergence of eigenfunctions}
\smallskip

Let $u_{\eps,h}$ be a normalized eigenfunction with uniformly bounded eigenvalue $\lambda_{\eps,h}$.
Since the harmonic extension of $\left. u_{\eps,h} \right|_{\Sigma \setminus B_{\eps^k}}$ to $\Sigma$
is uniformly bounded in $W^{1,2}(\Sigma)$ thanks to \cite[p.\ 40]{rt}, we get from the compact Sobolev embedding
subsequential convergence 
$$\widetilde{\left. u_{\eps,h} \right|_{\Sigma \setminus B_{\eps^k}}} \to \phi
$$ weakly in $W^{1,2}(\Sigma)$ and strongly in $L^2(\Sigma)$.
Since $C_c^\infty (\Sigma \setminus B_{\eps^k}) \subset W^{1,2}(\Sigma)$ is dense the weak convergence easily 
implies that $\phi$ either vanishes identically or is a non-trivial eigenfunction with eigenvalue $\lim_{\eps \to 0} \lambda_{\eps,h}$.
From the pointwise bound, the maximum principle and strong convergence in $L^2(\Sigma)$ we find that
$\|u_{\eps,h}\|_{L^2(\Sigma \setminus B_{\eps^k})} \to \|\phi\|_{L^2(\Sigma)}$.

If $\psi_{\eps,h,l}$ is a normalized $\lambda_l(M_{\eps,h})$-Dirichlet eigenfunction on $M_{\eps,h}$, we can test the corresponding eigenvalue equation against
$u_{\eps,h} - v_{\eps,h} \in W_0^{1,2}(M_{\eps,h})$
and find that
\begin{equation*}
\begin{split}
\lambda_l(M_{\eps,h,l}) \int_{M_{\eps,h}} \psi_{\eps,h,l} (u_{\eps,h} - v_{\eps,h})
&=
\int_{M_{\eps,h}} \nabla \psi_{\eps,h,l} \nabla(u_{\eps,h} - v_{\eps,h})
\\
&=
\lambda_{\eps,h} \int_{M_{\eps,h}} \psi_{\eps,h,l} u_{\eps,h}.
\end{split}
\end{equation*}
This implies
\begin{equation} \label{two_identity}
(\lambda_l(M_{\eps,h}) - \lambda_{\eps,h}) \int_{M_{\eps,h}} u_{\eps,h} \psi_{\eps,h,l} 
= 
\lambda_l(M_{\eps,h}) \int_{M_{\eps,h}} v_{\eps,h} \psi_{\eps,h,l}.
\end{equation}
Note that the Dirichlet spectrum of $M_{\eps,h}$ is simple and uniformly separated below any $\Lambda>0$ for $h \in [h_0,h_1]$ provided $\eps$ is sufficiently small (depending on $h_0,h_1,\Lambda$)
Therefore, the computation above implies thanks to \eqref{eq_est_harm} and H{\"o}lder's inequality, that, up to taking a subsequence, there can be at most one $l_*$ such that the integral on the left hand side of \eqref{two_identity} does not limit to zero.
By taking the square in \eqref{two_identity} and using again the uniform separation of the spectrum, we find for $l \neq l_*$ that
\begin{equation*}
\begin{split}
\left( \int_{M_{\eps,h}} u_{\eps,h} \psi_{\eps,h,l} \right)^2 
&\leq  \left( \frac{\lambda_l(M_{\eps,h})}{\lambda_l(M_\eps,h)-\lambda_{\eps,h}}\right)^2 \left( \int_{M_{\eps,h}} v_{\eps,h} \psi_{\eps,h,l} \right)^2
\\
& \leq
\left( 1+\frac{\Lambda}{c} \right)^2 \left( \int_{M_{\eps,h}} v_{\eps,h} \psi_{\eps,h,l} \right)^2.
\end{split}
\end{equation*}
Since the Dirichlet eigenfunctions form an orthonormal basis of $L^2(M_{\eps,h})$ this implies thanks to the pointwise bound \eqref{eq_est_harm} that
\begin{equation*}
\begin{split}
\int_{M_{\eps,h}}\left| u_{\eps,h} - \left( \int_{M_{\eps,h}} u_{\eps,h} \psi_{\eps,h,l_*} \right) \psi_{\eps,h,l_*}\right|^2
&=
\sum_{l \neq l_*} \left( \int_{M_{\eps,h}} u_{\eps,h} \psi_{\eps,h,l} \right)^2
\\
& \leq
\sum_{l \neq l_*} \left( 1+\frac{\Lambda}{c} \right)^2 \left( \int_{M_{\eps,h}} v_{\eps,h} \psi_{\eps,h,l} \right)^2
\\
&\leq  
\left( 1+\frac{\Lambda}{c} \right)^2 \int_{M_{\eps,h}}|v_{\eps,h}|^2 
\\
& \leq 
C \left( 1+\frac{\Lambda}{c} \right)^2 \eps \log(1/\eps^k)
\end{split}
\end{equation*}
uniformly in $h \in [h_0,h_1]$.
\end{proof}

We still need to provide the proof of \cref{max_eigen_neumann_conv}.

\begin{proof}[Proof of \cref{max_eigen_neumann_conv}]
The asymptotic upper bound on the eigenvalues follows from the same cut-off argument used in the first step above (cf.\ \eqref{eq_upper_bd}).
The functions $\tilde u_{\eps} \in W^{1,2}(\Sigma)$ are uniformly bounded in $W^{1,2}(\Sigma)$ by \cite[p.\ 40]{rt}.
Therefore, using that $C_c^\infty(\Sigma \setminus \{x_0\}) \subset W^{1,2}(\Sigma)$ is dense,
the asymptotic lower bound is a straightforward consequence of a standard compactness argument combined with the compact
Sobolev embedding on $\Sigma$ as in the third step above.
The assertion concerning the convergence of the eigenfunctions follows from the arguments above, combined with the maximum principle, and \cref{neu_ell_est}.
\end{proof}

\section[Quasimodes]{Construction of quasimodes}\label{sec4}

In this section we first briefly discuss the spectrum and the eigenfunctions of 
the cross cap attached to $\Sigma$ for the construction of $\Sigma_{\eps,h}$.
Afterwards, 
we construct various different types of quasimodes, i.e.\ approximate eigenfunctions.
These can be used to approximately locate eigenvalues and functions.
Denote by $(u_{\eps,h,l})_{l \in \IN}$ an orthonormal basis of eigenfunctions on $\Sigma_{\eps,h}$.

\begin{lemma}[cf.\ {\cite[Proposition $1$]{anne_3}}] \label{lem_anne_quasimod}
For any $\Lambda>0$, 
there is a uniform constant $C>0$ with the following property.
Let $f \in W^{1,2}(\Sigma_{\eps,h})$ be a function with $1/2 \leq \|f\|_{L^2(\Sigma_{\eps,h})}\leq 2$ such that
$$
\left| 
\int_{\Sigma_{\eps,h}} \nabla f \nabla \varphi - \lambda \int_{\Sigma_{\eps,h}}  f \varphi \,
\right|
\leq \delta \|\varphi\|_{W^{1,2}(\Sigma_{\eps,h})}
$$
for some $\delta>0$ and any $\varphi \in W^{1,2}(\Sigma_{\eps,h})$, where $\lambda \leq \Lambda$.
Let $0<s<1$ and write 
$$
g=\sum_{\{ l \ \colon \ |\lambda_l(\Sigma_{\eps,h})-\lambda| > s \}} \langle f , u_{\eps,h,l} \rangle_{L^2(\Sigma_{\eps,h})} u_{\eps,h,l}.
$$
Then
\begin{equation} \label{eq_quasimod_est}
\int_{\Sigma_{\eps,h}} |g|^2 + \int_{\Sigma_{\eps,h}} |\nabla g|^2 \leq C \frac{\delta^2}{s^2}.
\end{equation}
\end{lemma}

For sake of completeness, we have included a proof in \cref{sec_anne}.

\begin{rem} \label{rem_anne_quasimod}
Note that \cref{lem_anne_quasimod} can in particular be used to locate eigenvalues as follows.
If the right hand side in the \eqref{eq_quasimod_est} is strictly smaller than $\|f\|_{L^2(\Sigma_{\eps,h})}$, there is at least one eigenvalue contained in 
$[\lambda -s ,\lambda+s]$.
\end{rem}

Starting from eigenfunctions of $\Sigma$, we can construct quasimodes having most of their $L^2$-norm concentrated on $\Sigma$.
On the other hand, extending the Dirichlet eigenfunction of $M_{\eps,h}$ respectively $C_{\eps,h}$ carefully onto $\Sigma$, we obtain
quasimodes with most of their $L^2$-norm concentrated on $M_{\eps,h}$ or $C_{\eps,h}$, respectively.

\subsection{Quasimodes concentrated on $\Sigma$ } \label{sec_quasimodes_1}

We now construct two types of quasimodes resembling $\lambda_1(\Sigma)$-eigenfunctions.
The second construction works only under the symmetry assumption from the second part of \cref{main_1}.

\subsubsection{The case of cross caps} 

We start with the construction of the  quasimodes concentrated on $\Sigma$, which we obtain by simply cutting off an $L^2$-normalized $\lambda_1(\Sigma)$-eigenfunction near the points at which we attach and extending to all of $\Sigma_{\eps,h}$ by zero.

\smallskip

Let $\eta \colon [1,2] \to [0,1]$ be a function with 
$\eta(1)=0$
and
$\eta(2)=1$.
We then define the cut-off function $\eta_\eps \colon \Sigma_{\eps,h} \to [0,1]$ by
\begin{equation} \label{eq_def_cut_off}
\eta_{\eps}=
\begin{cases}
1 & \ \text{in} \  \Sigma \setminus B_{2 \eps^k}(x_0)
\\
\eta(\eps^{-k} r) & \ \text{in} \  B_{2 \eps^k} (x_0)
\\
0 & \ \text{on} \ M_{\eps,h},
\end{cases}
\end{equation}
where we use (Euclidean) radial coordinates $(\theta, r)$ in $B_{2 \eps^k}$.
Analogously, one can construct a cut-off function if $\Sigma_{\eps,h}$ is obtained by attaching $C_{\eps,h}$ which cuts-off near $x_0$ and $x_1$ using $\eta$.
By abuse of notation, we denote this function by $\eta_\eps$ as well.

\smallskip

For given $L^2$-normalized $\lambda_1(\Sigma)$-eigenfunction $\phi$, we define a new function by
\begin{equation} \label{eq_quasimod_surf}
\phi_\eps=
\begin{cases}
\eta_\eps \phi + (1-\eta_\eps) \phi(x_0) & \ \text{in} \ B_{2\eps^k}(x_0) \setminus B_{\eps^k}(x_0)
\\
\phi(x_0) & \ \text{on} \ M_{\eps,h}.
\end{cases}
\end{equation}
We will see below that if $\phi(x_0)=0$, the function $\phi_\eps$ turns out to be a good quasimode.
However, before we can actually prove this, we need to recall the following observation.

\begin{lemma} \label{mono_unif_sobolev}
Let $1<p<\infty$, then there is $C_p$ independent of $\eps$ and $k$, such that
\begin{equation*}
\| \varphi \|_{L^p(\Sigma \setminus B_{\eps^k})}
\leq C_p \|\varphi\|_{W^{1,2}(\Sigma \setminus B_{\eps^k})}
\end{equation*}
for any $\varphi \in W^{1,2}(\Sigma_{\eps,h})$.
\end{lemma}
\begin{proof}
This follows since the harmonic extension operator $W^{1,2}(\Sigma \setminus B_{\eps^k}) \to W^{1,2}(\Sigma)$
is uniformly bounded.
See e.g.\ \cite{rt}, where this is proved by a scaling argument.
The conclusion then follows by combining this with the Sobolev embedding $W^{1,2}(\Sigma) \hookrightarrow L^p(\Sigma)$.
\end{proof}

\begin{lemma} \label{mono_quasimode_1}
Let $\phi$ be an $L^2$-normalized $\lambda_1(\Sigma)$-eigenfunction.
We have for
the function $\phi_\eps$ defined above and any $\varphi \in W^{1,2}(\Sigma_{\eps,h})$, that
\begin{equation*}
\left| \int_{\Sigma_{\eps,h}} \nabla \phi_\eps \cdot \nabla \varphi - \lambda_1(\Sigma) \int_{\Sigma_{\eps,h}} \phi_\eps \varphi 
+ \lambda_1(\Sigma) \phi(x_0) \int_{M_{\eps,h}} \varphi
\, \right|
\leq
C \eps^{k/2} \| \varphi\|_{W^{1,2}(\Sigma_{\eps,h})}.
\end{equation*}
\end{lemma}

\begin{proof}
We compute
\begin{equation} \label{mono_part_int}
\begin{split}
\int_{\Sigma_{\eps,h}} \nabla \phi_\eps \cdot \nabla \varphi
=&
\int_{\Sigma_{\eps,h}} \nabla \phi \cdot \nabla (\eta_\eps \varphi) 
-
\int_{\Sigma_{\eps,h}} \varphi \nabla \phi \cdot \nabla \eta_\eps
+
\int_{\Sigma_{\eps,h}} (\phi-\phi(x_0)) \nabla \eta_\eps \cdot \nabla \varphi
\\
=&
\lambda_1(\Sigma )\int_{\Sigma_{\eps,h}} \phi \eta_\eps \varphi
-
\int_{\Sigma_{\eps,h}} \varphi \nabla \phi \cdot \nabla \eta_\eps
+
\int_{\Sigma_{\eps,h}} (\phi-\phi(x_0)) \nabla \eta_\eps \cdot \nabla \varphi
\\
=&
\lambda_1(\Sigma )\int_{\Sigma_{\eps,h}}\phi_\eps \varphi
-
\lambda_1(\Sigma) \phi(x_0) \int_{M_{\eps,h}} \varphi
-
\lambda_1(\Sigma) \phi(x_0) \int_{\Sigma \setminus B_{\eps^k}} (1-\eta_\eps) \varphi
\\
&-
\int_{\Sigma_{\eps,h}} \varphi \nabla \phi \cdot \nabla \eta_\eps
+
\int_{\Sigma_{\eps,h}} (\phi-\phi(x_0)) \nabla \eta_\eps \cdot \nabla \varphi,
\end{split}
\end{equation}
since $\eta_\eps \phi \in W^{1,2}(\Sigma)$.
Let us estimate the three last terms separately.
The first of these is small by H{\"o}lder's inequality,
\begin{equation}
\begin{split}
\left| \lambda_1(\Sigma) \phi_0(x_0) \int_{\Sigma \setminus B_{\eps^k}} (1-\eta_\eps) \varphi \right|
&\leq
C  \area(B_{2\eps^k})^{1/2} \left( \int_{\Sigma \setminus B_{\eps^k}} |\varphi|^2 \right)^{1/2}
\\
&\leq
C \eps^k \|\varphi\|_{L^2}(\Sigma_{\eps,h}).
\end{split}
\end{equation}
For the second term, we proceed as follows:
Since $\phi$ is smooth, there is a constant $C$ such that $|\nabla \phi| \leq C$.
Therefore, we can invoke  H{\"o}lder's inequality, the scaling invariance of the Dirichlet energy, and \cref{mono_unif_sobolev} to find
\begin{equation} \label{mono_term_1}
\begin{split}
\left| \int_{\Sigma_{\eps,h}} \varphi \nabla \phi \cdot \nabla \eta_\eps \, \right|
& \leq 
C \left( \int_{\Sigma_{\eps,h}}  |\varphi|^p \right)^{1/p} 
\area(B_{2\eps^k})^{1/q}
\left( \int_{\Sigma_{\eps,h}}  |\nabla \eta_\eps|^2 \right)^{1/2}
\\
& \leq C_p \eps^{2k/q} \| \varphi \|_{W^{1,2}(\Sigma_{\eps,h})},
\end{split}
\end{equation}
since it suffices to integrate over $\supp \nabla \eta_\eps \subset B_{2\eps^k}\setminus B_{\eps^k}$ and $1/p+1/q=1/2$.

\smallskip

We now estimate the last term from \eqref{mono_part_int}.
Since $\phi$ is smooth, there is a constant $C$, such that
\begin{equation*}
|\phi-\phi(x_0)| \leq C \eps^k
\end{equation*}
in $B_{2\eps^k}$.
Since $\supp \nabla \eta_\eps \subset B_{2\eps^k}$, this implies
\begin{equation} \label{mono_term_2}
\begin{split}
\left| \int_{\Sigma_{\eps,h}} (\phi-\phi(x_0)) \nabla \eta_\eps \cdot \nabla \varphi \, \right|
&\leq 
C \eps^k \left(\int_{\Sigma_{\eps,h}} |\nabla \eta_\eps|^2 \right)^{1/2}
\left(\int_{\Sigma_{\eps,h}} |\nabla \varphi|^2 \right)^{1/2}
\\
& \leq C \eps^k \|\varphi\|_{W^{1,2}(\Sigma_{\eps,h})}
\end{split}
\end{equation}
by H{\"o}lder's inequality and the scaling invariance of the Dirichlet energy.
If we specify to $p=q=4$ in \eqref{mono_term_1} and combine this with \eqref{mono_part_int} and \eqref{mono_term_2}, the assertion follows.
\end{proof}

\subsubsection{The case of cylinders under symmetry assumption} 

Under the symmetry assumption that
\begin{equation} \label{eq_eigen_sym}
\phi(x_0)+\phi(x_1)=0
\end{equation}
for any $\lambda_1(\Sigma)$-eigenfunction there is a more sensitive way to extend $\phi$ across $C_{\eps,h}$
at least if $h$ is close to $h_*$. (Recall that $h_*$ is such that $\lambda_0(C_{\eps,h_*})=\lambda_1(\Sigma)$.)
The starting point for this construction is the observation that the eigenvalues $\lambda_0(C_{\eps,h})$ with Dirichlet boundary conditions and
$\mu_1(C_{\eps,h})$ with Neumann boundary conditions agree if $\eps$ is sufficiently small.
Moreover, for such $\eps$, any $\mu_1(C_{\eps,h})$-eigenfunction is antisymmetric with respect to the involution $(\theta,t) \mapsto (\theta,h-t)$.
Thus, we can hope to find a good quasimode by interpolating from $\phi(x_0)$ to $\phi(x_1)$ by a $\mu_1(C_{\eps,h})$-eigenfunction on $C_{\eps,h}$.

\smallskip

To make this precise, given a $\lambda_1(\Sigma)$-eigenfunction with \eqref{eq_eigen_sym} we define a function $\phi_\eps^N \in W^{1,2}(\Sigma_{\eps,h})$ as follows
\begin{equation*}
\phi_\eps^N=
\begin{cases}
\phi & \ \text{in} \ \Sigma \setminus B_{2\eps^k}(x_0) \cup B_{2\eps^k} (x_1)
\\
\eta_\eps\phi + (1-\eta_\eps) \phi(x_0) & \ \text{in} \ B_{2\eps^k} \setminus B_{\eps^k}(x_0)
\\
\eta_\eps \phi + (1-\eta_\eps )\phi(x_1) & \ \text{in} \ B_{2\eps^k} \setminus B_{\eps^k}(x_1)
\\
\psi^N & \ \text{on} \ C_{\eps,h},
\end{cases}
\end{equation*}
where $\psi^N \colon C_{\eps,h} \to \IR$ is a $\mu_1(C_{\eps,h})$-eigenfunction that is equal to $\phi(x_0)$ respectively $\phi(x_1)$ on 
the boundary components of $C_{\eps,h}$.
Note that such a $\psi^N$ exists precisely since we assume that $\phi$ satisfies \eqref{eq_eigen_sym}.

\smallskip

For $h$ close to $h_*$ the function $\phi_\eps^N$ provides a good quasimode as demonstrated below.
For the proof of \cref{main_2} it is important to carefully keep track of the dependence of the estimate on the parameter $h$.

\begin{lemma} \label{mono_quasimode_2}
For the function $\phi_\eps^N$ defined above and any $\varphi \in W^{1,2}(\Sigma_{\eps,h})$, we have that
\begin{equation*}
\left| \int_{\Sigma_{\eps,h}} \nabla \phi_\eps^N \cdot \nabla \varphi - \lambda_1(\Sigma) \int_{\Sigma_{\eps,h}} \phi_\eps^N \varphi \, \right|
\leq
C \left( \left| \frac{1}{h^2} - \frac{1}{h_*^2}  \right| \eps^{1/2} + \eps^{k/2}\right) \| \varphi \|_{W^{1,2}(\Sigma_{\eps,h})}
\end{equation*}
\end{lemma}

\begin{proof}
The estimate in $\Sigma \setminus B_{\eps^k}$ (where $B_{\eps^k}=(B(x_0,\eps^k) \cup B(x_1,\eps^k))$ carries over mutatis mutandis from the proof of \cref{mono_quasimode_1} and implies
\begin{equation*}
\left|  \int_{\Sigma \setminus B_{\eps^k}} \nabla \phi_\eps^N \cdot \nabla \varphi - \lambda_1(\Sigma) \int_{\Sigma \setminus B_{\eps^k}} \phi_\eps^N \varphi \, \right|
\leq
C_p \eps^{2k/q} \| \varphi \| _{W^{1,2}(\Sigma_{\eps,h})}
\end{equation*}
where $1/p+1/q=1/2$.

\smallskip

On the cylinder we have that
\begin{equation*}
\begin{split}
\left|  \int_{C_{\eps,h}} \nabla \phi_\eps^N \cdot \nabla \varphi - \lambda_1(\Sigma)  \int_{C_{\eps,h}} \phi_\eps^N \phi \, \right|
&=
\left|  \int_{C_{\eps,h}} \nabla \psi^N \cdot \nabla \varphi - \lambda_1(\Sigma)  \int_{C_{\eps,h}} \psi^N \varphi \, \right|
\\
& \leq
 |\mu_1(C_{\eps,h}) -\lambda_1(\Sigma)| \int_{C_{\eps,h}} |\psi^N \varphi |
  \\
& \leq 
C \left| \frac{1}{h^2} - \frac{1}{h_*^2} \right| \int_{C_{\eps,h} }|\varphi |
\\
& \leq C \left| \frac{1}{h^2} - \frac{1}{h_*^2}  \right| \eps^{1/2} \| \varphi \|_{W^{1,2}(\Sigma_{\eps,h})}.
\end{split}
\end{equation*}
Combining the above two estimates and specifying to $p=q=4$ implies the assertion.
\end{proof}

\subsection{Quasimodes concentrated on a handle or cross cap}  \label{sec_quasi_handle}

In this subsection we construct a quasimode from the first Dirichlet eigenfunction of the handle $C_{\eps,h}$ or cross cap $M_{\eps,h}$, respectively.
The naive choice of simply extending a Dirichlet eigenfunction to all of $\Sigma_{\eps,h}$ by $0$ turns out to be not good enough.
In order to obtain a good quasimode we need to find a good extension of the normalized first Dirichlet eigenfunction to $\Sigma \setminus B_{\eps^k}$.
In principal one would like to use the Green function of $\Delta-\lambda$ with pole at $x_0$.
While this works very well for a \emph{fixed} choice of the parameter $h$, we need to be more careful when considering the whole family $\Sigma_{\eps,h}$.
The presence of a non-trivial kernel of $\Delta - \lambda_0(M_{\eps,h})$ for $h = h_*$ forces us to modify
the Green function also for $h$ close to $h_*$ in order to make our estimates uniform.

\subsubsection{The first eigenfunction of $M_{\eps,h}$ and $C_{\eps,h}$}

Since we will only use the first Dirichlet eigenfunction of $M_{\eps,h}$ from here on, we simply denote it by $\psi_{\eps,h}$
instead of $\psi_{\eps,h,0}$.
A direct computation immediately gives that $\psi_{\eps,h}$ is explicitly given by
\begin{equation*} 
\psi_{\eps,h} = \eps^{-1/2}\psi_h= \frac{1}{\sqrt{\pi \eps h/2}} \sin\left(\frac{t \pi}{ h} \right) 
\end{equation*}
parametrized on the covering space $S^1(\eps) \times [0,h]$.
For the $L^1$-norm, we have that
\begin{equation} \label{eq_l1_norm}
\int_{M_{\eps,h}} \psi_{\eps,h}
=
4 \left( \frac{h}{2 \pi}\right)^{1/2} \eps^{1/2}.
\end{equation}
Finally, for the normal derivative, we get that
$$
\int_{\partial M_{\eps,h}} \partial_\nu \psi_{\eps,h} d \mathcal{H}^1
=
-2 \pi \eps \frac{\pi}{h \sqrt{\pi \eps h/2}}
=
-\left(\frac{2 \pi}{h}\right)^{3/2} \eps^{1/2}.
$$
It is the scaling of this term in $\eps$ combined with the
presence of a non-trivial kernel of $\Delta - \lambda_1(\Sigma)$ that forces the order of the leading order term in \cref{main_2} to be on scale $\eps^{1/2}$.

Let us briefly discuss what happens for the naive quasimode given by extending $\psi_{\eps,h}$ to all of $\Sigma_{\eps,h}$ by $0$.
Thanks to \eqref{eq_int_bd_sobolev} integration by parts\footnote{Note that we integrate with respect to the Hausdorff measure of $\partial M_{\eps,h}$ and not of $\partial B_{\eps^k}$ below} implies that
\begin{equation}
\begin{split}
\left| \int_{\Sigma_{\eps,h}} \nabla \varphi \cdot \nabla \psi_{\eps,h} - \lambda_0(M_{\eps,h}) \int_{\Sigma_{\eps,h}} \varphi \psi_{\eps,h} \right|
&\leq 
\int_{\partial M_{\eps,h}} |\varphi| |\partial_\nu \psi_{\eps,h}|\, d \mathcal{H}^1
\\
&\leq
C \eps^{1/2} \log(1/\eps^k) \|\varphi\|_{W^{1,2}(\Sigma_{\eps,h})}.
\end{split}
\end{equation}
Moreover, this bound is easily seen to be sharp. 

The extension of $\psi_{\eps,h}$ to $\Sigma \setminus B_{\eps^k}$ constructed below cancels out the normal derivative along $\partial B_{\eps^k}$.
This has two advantages over the naive quasimode $\psi_{\eps,h}$.
Firstly, it is an approximate solution on a strictly smaller scale.
Secondly, we can identify the largest error term very precisely using the convergence result on the eigenfunctions from \cref{conv}.

The very same discussion applies to the first Dirichlet eigenfunction on $C_{\eps,h}$ in this case the normalized eigenfunction
is given by $1/ (\sqrt{2} \sqrt{\pi\eps h/2}) \sin (t \pi /h)$ and 
we have that
$$
\int_{\partial C_{\eps,h}} \partial_\nu \frac{1}{\sqrt{\pi \eps h}} \sin \left( \frac{t \pi}{h} \right) d \mathcal{H}^1
=
- 4 \left( \frac{\pi}{h} \right)^{3/2} \eps^{1/2}.
$$

\subsubsection{The Green's function of $(\Delta_\Sigma - \lambda)$}

We need some preliminaries on a function closely related to the Green's function of the operator $\Delta-\lambda$ on $\Sigma$.
For the convenience of the reader, the short \cref{sec_green} contains a proof of the facts on Green's functions that we make use of below.
Recall that if we normalize $\area(\Sigma)=1$ the Green's function $G(x,y)$ of $\Delta$ solves

\begin{equation*}
\Delta_y G(x,y)=\delta_x - 1.
\end{equation*}
in the sense of distributions.
Near the diagonal, the Green's function is asymptotic to the Green's function of the Euclidean plane.
More precisely, for $x \in \Sigma$ fixed, we have that
\begin{equation} \label{green_asymp_1}
G(x,y)=\frac{1}{2 \pi} \log\left(\frac{1}{|x-y|}\right) + \psi_x(y),
\end{equation}
where $|x-y|$ is the distance with respect to the Euclidean metric in conformal coordinates near $x$ normalized such that $g=fg_e$ with $f(x)=1$ and $\psi_x$ is a smooth function.
Off the diagonal, $G$ is a smooth function.
In particular, we find that
\begin{equation} \label{integral_green_bd}
\int_\Sigma |G(x,y)|^p dy \leq C
\end{equation}
for any $p < \infty$ and some uniform constant $C=C(\Sigma,p)$.

Let $(\phi_0,\dots,\phi_{K-1})$ be an orthonormal basis of the $\lambda_1(\Sigma)$-eigenspace.
We consider the function
\begin{equation*}
f(y)=G(x_0,y)-\sum_{i=0}^{K-1} \int_\Sigma G(x_0,z) \phi_i(z)dz \, \phi_i(y),
\end{equation*}
which is well-defined by H{\"o}lder's inequality and \eqref{integral_green_bd}.
Also from \eqref{integral_green_bd} and H{\"o}lder's inequality, we find that
\begin{equation} \label{f_lp_bd}
\int_\Sigma |f|^p \leq C,
\end{equation}
for a constant $C=C(\Sigma,p)$.
In particular, for any $\lambda \in (0,\lambda_{K+1}(\Sigma))$ there is unique solution $u_\lambda \in W^{1,2}(\Sigma)$ that is orthogonal to $\langle \phi_0, \dots , \phi_{K-1} \rangle$ and such that
\begin{equation*}
(\Delta-\lambda)u_\lambda = \lambda f+1
\end{equation*}
since $f$ and the constant functions are orthogonal to the kernel and hence also the cokernel of $(\Delta-\lambda)$
(which for the relevant $\lambda$ is trivial if $\lambda \neq \lambda_1(\Sigma)$ and  equal to $\langle \phi_0, \dots , \phi_{K-1} \rangle$ if $\lambda = \lambda_1(\Sigma)$).
It follows from \eqref{f_lp_bd} and standard elliptic estimates that $u_\lambda$ is uniformly bounded in $W^{2,p}(\Sigma)$ as long as $\lambda \in [\delta_0 ,\lambda_{K+1}(\Sigma)-\delta_0]$ for some small $\delta_0>0$.
We now fix $\delta_0>0$ once and for all such that 
\begin{equation}
\label{delta}
\delta_0 < \lambda_1(\Sigma) < \lambda_{K+1}(\Sigma)-\delta_0.
\end{equation}
It is convenient to make some more restrictions on $h_0$ at this point.
In addition to \eqref{eq_choice_h}, we also assume that
\begin{equation} \label{eq_choice_h_0}
\lambda_0(C_{\eps,h_0})=\lambda_0(M_{\eps,h_0})=\frac{\pi}{h_0^2} \leq \lambda_{K+1}(\Sigma) -\delta_0.
\end{equation}
The Sobolev embedding theorem yields that $u_\lambda$ is uniformly bounded in $C^{1,\alpha}(\Sigma)$ for some $\alpha>0$ provided we choose $p>2$ above.
Consider the function 
\begin{equation*}
H_\lambda(y)=G(x_0,y)+u_\lambda(y).
\end{equation*}
If we choose the orhonormal basis of $\lambda_1(\Sigma)$-eigenfunctions such that
\begin{equation} \label{eq_normal_basis}
\phi_1(x_0)= \dots =\phi_{K-1}(x_0)=0,
\end{equation}
we find that $H_\lambda$ solves
\begin{equation} \label{eq_green_proj_kernel}
\begin{split}
(\Delta - \lambda) H_\lambda
&= 
- \lambda \sum_{i=0}^{K-1}\int_\Sigma G(x_0,z) \phi_i(z)dz \, \phi_i
\\
& =
- \frac{\lambda}{\lambda_1(\Sigma)} \phi_0(x_0)\phi_0
\end{split}
\end{equation}
in $\Sigma \setminus \{x_0\}$ by the normalization \eqref{eq_normal_basis}.

Since $u$ is uniformly bounded in $C^{1,\alpha}(\Sigma)$ for a fixed $\alpha>0$, we find from \eqref{green_asymp_1} that the function
\begin{equation*}
e_{\lambda}(y):=H_\lambda(y)- \frac{1}{2 \pi} \log\left(\frac{1}{|x_0-y|}\right) 
\end{equation*}
is uniformly bounded in $C^{1,\alpha}(\Sigma)$.
Therefore, 
\begin{equation*}
e_{\eps,\lambda}:= \frac{2\pi e_\lambda(x_0)}{\log(1/\eps^k)}=o(1)
\end{equation*}
as $\eps \to 0$ uniformly in $\lambda \in [\delta_0,\lambda_K+1(\Sigma)-\delta_0]$.

Denote by $H_{\lambda,0}$ the function constructed above with a pole at $x_0$ and by $H_{\lambda,1}$ the analogously constructed function
with a pole at $x_1$. 
Similarly, we write $e_{\lambda,0}$ and $e_{\lambda,1}$ for the corresponding terms in the asymptotic expansion of $H_{\lambda,0}$ and $H_{\lambda,1}$, respectively.
Consider 
$$
J_\lambda = H_{\lambda,0} + H_{\lambda,1},
$$
which has poles at $x_0$ and $x_1$.

If we choose the orthonormal basis $(\phi_0,\dots,\phi_{K-1})$ of $\lambda_1(\Sigma)$-eigenfunctions such that
$$
\phi_i(x_0) + \phi_i(x_1) =0
$$
for $i=1,\dots,K-1$ we find similarly as in
\eqref{eq_green_proj_kernel} 
that
\begin{equation} \label{eq_kernel_handle}
(\Delta - \lambda) J_\lambda =  - \frac{\lambda}{\lambda_1(\Sigma)} (\phi_0(x_0)+\phi_0(x_1))\phi_0.
\end{equation}
In particular, if $x_0$ and $x_1$ are two points in $\Sigma$ such that 
\begin{equation} \label{eq_assum_key_cancel}
\phi(x_0) + \phi(x_1) = 0
\end{equation}
for any $\lambda_1(\Sigma)$-eigenfunction $\phi$
we have that
\begin{equation} \label{eq_sym_green}
(\Delta - \lambda) J_{\lambda} = 0
\end{equation}
in $\Sigma \setminus \{x_0,x_1\}$.
This and \cref{mono_quasimode_2} are the two reasons for the good control of the first eigenvalue in the second part of \cref{main_1}.

\subsubsection{Construction of the quasimodes for cross caps.} \label{sub_sec_quasi_cross}

Recall the definition of the cut-off functions $\eta_\eps \colon \Sigma \setminus B_{\eps^k} \to  [0,1]$ defined at \eqref{eq_def_cut_off}
and that we write
$$
\psi_{\eps,h}=\eps^{-1/2}\psi_h
$$
 the ground state of $M_{\eps,h}$.
We define $\tilde \psi_{\eps,h} \in W^{1,2}(\Sigma_{\eps,h})$ as follows,
\begin{equation*}
\tilde \psi_{\eps,h}(y)=
\begin{cases}
 \left(\frac{2 \pi}{h}\right)^{3/2}\eps^{1/2} \left(\eta_\eps H_\lambda(y) + (1-\eta_\eps)   \left(\frac{1}{2\pi}\log \left( \frac{1}{|x_0-y|} \right) + e_{\lambda}(x_0) \right) \right) & \text{on} \ \Sigma \setminus B_{\eps^k}
\\
\psi_{\eps,h}(y) + \left(\frac{2\pi}{h}\right)^{3/2}(1+e_{\eps,\lambda}) \frac{1}{2\pi}\log\left({1}/{\eps^k}\right) \eps^{1/2} & \text{on} \ M_{\eps,h},
\end{cases}
\end{equation*}
where $\lambda=\lambda_0(M_{\eps,h})$. 
By construction, $\tilde \psi_{\eps,h}$ is a Lipschitz function, in particular we have $\tilde \psi_{\eps,h} \in W^{1,2}(\Sigma_{\eps,h})$.
The key property of $\tilde \psi_{\eps,h}$ is that it is rotationally symmetric near $\partial B_{\eps^k}$ and has
\begin{equation*}
\begin{split}
\int_{\partial (\Sigma \setminus B_{\eps^k})} \partial_{\nu_{\it{eucl}}} \tilde \psi_{\eps,h}d\mathcal{H}^1_{\it{eucl}}
&=
\frac{1}{2\pi} \left(\frac{2 \pi}{h} \right)^{3/2} \eps^{1/2} \int_{\partial B_{\eps^k}} \partial_r \log(r) d \mathcal{H}^1_{\it{eucl}}
\\
&=
\left(\frac{2 \pi}{h} \right)^{3/2} \eps^{1/2}
\\
&=
-\int_{\partial M_{\eps,h}} \partial_{\nu} \psi_{\eps,h}d\mathcal{H}^1.
\end{split}
\end{equation*}
Here and also below we use the convention that the domain of integration also indicates which normal and measure we use.
This is particularly important along $\partial M_{\eps,\kappa}=\partial B_{\eps^k}$, where these differ significantly.

Let $u_{\eps,h}$ be an $L^2(\Sigma_{\eps,h})$-normalized eigenfunction on $\Sigma_{\eps,h}$ with eigenvalue $\lambda_{\eps,h} \in (\delta_0,\lambda_{K+1}(\Sigma)-\delta_0)$, such that
\begin{equation} \label{eq_assum_conc_cyl}
\int_{M_{\eps,h}} |u_{\eps,h}|^2 \geq c_0>0
\end{equation}
 for some fixed constant $c_0$.
Thanks to the last part of \cref{conv} and recalling our choice of $h_0$ at \eqref{eq_choice_h_0} this implies that
\begin{equation*}
\left| \int_{M_{\eps,h}} \psi_{\eps,h} u_{\eps,h} \right| \geq 2c_1
\end{equation*}
for some uniform $c_1=c_1(c_0,h_0,h_1)>0$, which in turn implies that
\begin{equation} \label{eq_sign}
\left| \int_{M_{\eps,h}} \tilde \psi_{\eps,h} u_{\eps,h} \right| \geq c_1
\end{equation}
for $\eps$ sufficiently small.

We now provide a first asymptotic expansion of $\lambda_{\eps,h}$.
Below, for simplicity, we write $\lambda = \lambda_0(M_{\eps,h})$.

\begin{lemma} \label{lem_asymp_1}
The eigenvalue $\lambda_{\eps,h}$ has the asymptotic expansion
\begin{equation} \label{eq_asymp_1}
\begin{split}
\lambda_{\eps,h}
&= 
\lambda  - \left(\frac{2 \pi}{h} \right)^{3/2} \frac{\frac{\lambda}{\lambda_1(\Sigma)} \int_{\Sigma \setminus B_{2\eps^k}} \phi_0(x_0) \phi_0 u_{\eps,h} \eps^{1/2}+ \frac{\lambda}{2\pi} (1+e_{\eps,\lambda}) \int_{M_{\eps,h}} u_{\eps,h} \log(1/\eps^k)\eps^{1/2}  }{  \int_{\Sigma_{\eps,h}}  u_{\eps,h} \tilde \psi_{\eps,h}}
\\
&+ O(\eps^{k}\log(1/\eps^k))
\end{split}
\end{equation}
as $\eps \to 0$, uniformly in $h \in [h_0,h_1]$ as long as $u_{\eps,h}$ satisfies \eqref{eq_assum_conc_cyl} with $c_0>0$ fixed.
\end{lemma}

\begin{rem} \label{rem_order}
Note that by H{\"o}lder's inequality,
\begin{equation*}
\begin{split}
\left |\int_{M_{\eps,h}} u_{\eps,h} \log(1/\eps^k)\eps^{1/2}  \right|
&\leq 
\area(M_{\eps,h})^{1/2}\eps^{1/2} \log(1/\eps^k)\|u_{\eps,h}\|_{L^2}
\\
& \leq
 C \eps \log(1/\eps^k)\|u_{\eps,h}\|_{L^2},
 \end{split}
\end{equation*}
so that the first summand in the enumerator in \eqref{eq_asymp_1} is the term of lower order. 
\end{rem}

\begin{proof}
Since $u_{\eps,h}$ is an eigenfunction and $\tilde \psi_{\eps,h} \in W^{1,2}(\Sigma_{\eps,h})$ we have that
\begin{equation*}
 \int_{\Sigma_{\eps,h}} \nabla u_{\eps,h} \nabla \tilde \psi_{\eps,h}
=
\lambda_{\eps,h}
\int_{\Sigma_{\eps,h}} u_{\eps,h} \tilde \psi_{\eps,h}.
\end{equation*}
On the other hand, we have that
\begin{equation*}
\begin{split}
 \int_{\Sigma_{\eps,h}} \nabla u_{\eps,h} \nabla \tilde \psi_{\eps,h}
 &=
 \int_{\Sigma \setminus B_{\eps^k}} \nabla u_{\eps,h} \nabla \tilde \psi_{\eps,h}
 + 
 \int_{M_{\eps,h}} \nabla u_{\eps,h} \nabla \tilde \psi_{\eps,h}
 \\
 &=
 \int_{\Sigma \setminus B_{\eps^k}} u_{\eps,h} \Delta \tilde \psi_{\eps,h} 
 + 
 \int_{\partial B_{\eps^k}} u_{\eps,h} \partial_\nu  \tilde \psi_{\eps,h} d \mathcal{H}^1
 \\
 &+
  \int_{M_{\eps,h}} u_{\eps,h} \Delta \tilde \psi_{\eps,h} 
  -
  \int_{\partial M_{\eps,h}} u_{\eps,h} \partial_{\nu} \tilde \psi_{\eps,h} d \mathcal{H}^1
  \\
  & =
   \int_{\Sigma \setminus B_{\eps^k}} u_{\eps,h} \Delta \tilde \psi_{\eps,h} 
   + \int_{M_{\eps,h}} u_{\eps,h} \Delta \tilde \psi_{\eps,h}
   + O(\eps^k).
\end{split}
\end{equation*}
The last step used that
\begin{equation*}
\begin{split}
& \left|  \int_{\partial B_{\eps^k}} u_{\eps,h} \partial_\nu \tilde \psi_{\eps,h} d \mathcal{H}^1 - \int_{\partial M_{\eps,h}} u_{\eps,h} \partial_{\nu} \tilde \psi_{\eps,h} d \mathcal{H}^1
\right| 
\\
&=
\left|
 \int_{\partial B_{\eps^k}} u_{\eps,h} \nabla \tilde \psi_{\eps,h} \cdot \nu d \mathcal{H}^1
 -\int_{\partial B_{\eps^k}} u_{\eps,h} \nabla \tilde \psi_{\eps,h} \cdot \nu_{\it{eucl}} d \mathcal{H}^1_{\it{eucl}}
\right|
\\
&\leq 
C \eps^{1/2} \log (1/\eps^k) \| g - g_{\it{eucl}}\|_{L^\infty(\partial B_{\eps^k})}
 \leq C \eps^k,
 \end{split}
\end{equation*}
where we have used that $|u_{\eps,h}| \leq C \log(1/\eps^k)$ on $\partial B_{\eps^k}$ (see \cref{max_eigen_neu_ell_est}).
Moreover, we have that
\begin{equation*}
\begin{aligned}
\int_{M_{\eps,h}}  u_{\eps,h} \Delta \tilde \psi_{\eps,h} =
\lambda \int_{M_{\eps,h}} u_{\eps,h} \tilde  \psi_{\eps,h} - \left(\frac{2 \pi}{h}\right)^{3/2}\frac{\lambda}{2\pi} (1+e_{\eps,\lambda}) \int_{M_{\eps,h}} u_{\eps,h} \log(1/\eps^{k}) \eps^{1/2}.
\end{aligned}
\end{equation*}
In order to estimate the integral on $\Sigma \setminus B_{\eps^k}$ we have that
\begin{equation*}
\begin{aligned}
\int_{\Sigma \setminus B_{\eps^k}} &u_{\eps,h} (\Delta - \lambda) \tilde \psi_{\eps,h}
\\
=&
-\eps^{1/2}\frac{\lambda}{\lambda_1(\Sigma)}\int_{\Sigma \setminus B_{2 \eps^k}} u_{\eps,h}\phi_0(x_0)\phi_0
+
\int_{B_{2 \eps^k} \setminus B_{\eps^k}} u_{\eps,h}(\Delta-\lambda) \tilde \psi_{\eps,h}.
\end{aligned}
\end{equation*}
thanks to \eqref{eq_green_proj_kernel}.
It remains to estimate the second summand.
First note that, using \eqref{eq_green_proj_kernel}, we have in $\Sigma \setminus B_{\eps^k}$ that
\begin{equation*}
\begin{split}
\left(\frac{2 \pi}{h} \right)^{-3/2}\eps^{-1/2}
&
(\Delta - \lambda) \tilde \psi_{\eps,h}
=
\eta_\eps(\Delta-\lambda)H_\lambda
 -
 2\nabla \eta_\eps \cdot \nabla H_\lambda
  +
 H_\lambda \Delta \eta_\eps
 \\
 &+
\frac{2}{2\pi} \nabla \eta_\eps \cdot \nabla \log \left( \frac{1}{|x_0-y|}\right)
-
\left( \frac{1}{2 \pi}  \log \left( \frac{1}{|x_0-y|}\right) + e_\lambda(x_0) \right) \Delta \eta_\eps
\\
& - 
\lambda (1-\eta_\eps) \left( \frac{1}{2\pi} \log \left( \frac{1}{|x_0-y|}\right) + e_\lambda(x_0) \right)
\\
&=
-\frac{\lambda}{\lambda_1(\Sigma)} \eta_\eps \phi_0(x_0)\phi_0   
  - 
 2\nabla \eta_\eps \cdot \left( \nabla H_\lambda- \frac{1}{2\pi} \nabla \log \left( \frac{1}{|x_0-y|} \right) \right)
 \\
 &+
 \Delta \eta_\eps \left(H_\lambda -  \frac{1}{2 \pi}  \log \left( \frac{1}{|x_0-y|}\right) - e_\lambda(x_0) \right)
 \\
& - 
\lambda (1-\eta_\eps) \left( \frac{1}{2\pi} \log \left( \frac{1}{|x_0-y|}\right) + e_\lambda(x_0) \right)
 \end{split}
\end{equation*}
since $\Delta \log(1/|x_0-y|)=0$ thanks to the conformal covariance of the Laplacian.
Therefore, we find that
\begin{equation*}
\begin{split}
 \int_{B_{2 \eps^k} \setminus B_{\eps^k}} &
 \left| u_{\eps,h} (\Delta - \lambda) \tilde \psi_{\eps,h}
\right|
 \leq
 C     \left( 
 \int_{B_{2 \eps^k} \setminus B_{\eps^k}}  |u_{\eps,h}||\eta_\eps \phi_0(x_0)\phi_0|
 + \right.
  \int_{B_{2 \eps^k} \setminus B_{\eps^k}} |u_{\eps,h}| |\nabla \eta_\eps| |\nabla e_\lambda|
  \\
  &+ \left.
   \int_{B_{2 \eps^k} \setminus B_{\eps^k}}  |u_{\eps,h}||\Delta \eta_\eps| |e_\lambda - e_\lambda(x_0)|
   + 
    \lambda \int_{B_{2 \eps^k} \setminus B_{\eps^k}}  |u_{\eps,h}| \log(1/\eps^k)
    \right) \eps^{1/2}
   \\
    \leq &
   C (\eps^{2k}\log(1/\eps^k) + \eps^{2k} \eps^{-k} \log(1/\eps^k) +  \eps^{2k} \log(1/\eps^k) \eps^{-2k} \eps^{k}
   + \log^2(1/\eps^k)\eps^{2k})\eps^{1/2}
   \\
   \leq & C \eps^k
\end{split}
\end{equation*}
for $\eps>0$ sufficiently small,
where we have used \cref{max_eigen_neu_ell_est} and $e_\lambda \in C^{1,\alpha}(\Sigma)$.
The assertion now follows from combining all the above estimates.
\end{proof}

Recall that we chose an orthonormal basis of $\lambda_1(\Sigma)$-eigenfunctions at \eqref{eq_normal_basis} such that all eigenfunctions orthogonal to
$\phi_0$ vanish at $x_0$.
We assume for the rest of this subsection that
$$
\phi_0(x_0) \neq 0.
$$
In this case we can define a second quasimode $\chi_{\eps,h}$ concentrated on $M_{\eps,h}$ slightly different from $\tilde \psi_{\eps,h}$.
It only differs from $\tilde \psi_{\eps,h}$ in the correction term necessary because of the different boundary values
of $\psi_{\eps,h}$ and $\log(1/|x_0-\cdot|)$ along $\partial M_{\eps,h}=\partial B_{\eps^k}(x_0)$.
We define
\begin{equation*}
\begin{aligned}
\chi_{\eps,h}(y)=
&\left(\frac{2 \pi}{h} \right)^{3/2} \eps^{1/2} \left(\eta_\eps H_\lambda(y) + (1-\eta_\eps)   \left(\frac{1}{2\pi}\log \left( \frac{1}{|x_0-y|} \right) + e_{\lambda}(x_0) \right) \right)
\\\hspace{1cm}&-\left(\frac{2 \pi}{h} \right)^{3/2}\frac{1}{2 \pi}(1+e_{\eps,\lambda})\log\left({1}/{\eps^k}\right) \eps^{1/2} \frac{\phi_{0,\eps}}{\phi_0(x_0)}
\end{aligned}
\end{equation*}
on  $\Sigma \setminus B_{\eps^k}$ and by $\chi_{\eps,h}=\psi_{\eps,h}$ on $M_{\eps,h}$.
Here, $\phi_{0,\eps}$ denotes the function constructed from $\phi_0$ in \eqref{eq_quasimod_surf}.
Note that $\phi_{0,\eps} = \phi_0(x_0)$ along $\partial B_{\eps^k}$ and $\phi_{0,\eps}=\phi_0$ outside of $B_{2\eps^k}$ by construction.

As indicated above, we have that
\begin{equation} \label{eq_diff_quasimod}
\tilde \psi_{\eps,h} - \chi_{\eps,h}
=
\begin{cases}
\left(\frac{2 \pi}{h} \right)^{3/2}\frac{1}{2 \pi}(1+e_{\eps,\lambda})\log\left({1}/{\eps^k}\right) \eps^{1/2} \frac{\phi_{0,\eps}}{\phi_0(x_0)} & \ \text{in} \ \Sigma \setminus B_{\eps^k}
\\
\left(\frac{2\pi}{h}\right)^{3/2}\frac{1}{2\pi}(1+e_{\eps,\lambda})\log\left({1}/{\eps^k}\right) \eps^{1/2} & \ \text{on}  \ M_{\eps,h}.
\end{cases}
\end{equation}

\smallskip

If we take an eigenfunction $u_{\eps,h}$ as in \eqref{eq_assum_conc_cyl} above, we have that
$$
\left| \int_{\Sigma_{\eps,h}} \chi_{\eps,h} u_{\eps,h} \right| \geq c_1
$$ 
for $\eps$ sufficiently small.
The corresponding asymptotic expansion for the associated eigenvalue $\lambda_{\eps,h}$ is as follows.

\begin{lemma} \label{lem_asymp_2}
The eigenvalue $\lambda_{\eps,h}$ has the asymptotic expansion
\begin{equation} \label{eq_asymp_2}
\begin{aligned}
\lambda_{\eps,h}
= 
\lambda &- \left(\frac{2 \pi}{h} \right)^{3/2}\frac{  \frac{\lambda}{\lambda_1(\Sigma)} \int_{\Sigma \setminus B_{2\eps^k}} \phi_0(x_0)\phi_0 u_{\eps,h} \eps^{1/2}}{ \int_{\Sigma_{\eps,h}}  u_{\eps,h} \chi_{\eps,h}} 
\\
&- \left(\frac{2 \pi}{h} \right)^{3/2} \frac{ \frac{\lambda_1(\Sigma)-\lambda}{2\pi} (1+e_{\eps,\lambda}) \int_{\Sigma \setminus B_{2\eps^k}} \frac{\phi_0}{\phi_0(x_0)}u_{\eps,h}  \eps^{1/2} \log(1/\eps^k) }{  \int_{\Sigma_{\eps,h}}  u_{\eps,h} \chi_{\eps,h}}
\\
&+ O(\eps^{k}\log(1/\eps^k))
\end{aligned}
\end{equation}
as $\eps \to 0$, uniformly in $h \in [h_0,h_1]$ as long as $u_{\eps,h}$ satisfies \eqref{eq_assum_conc_cyl} with $c_0>0$ fixed.
\end{lemma}

\begin{proof}
This is completely analogous to the proof of \cref{lem_asymp_1}.
The form of the second summand in the enumerator stems from the fact that
\begin{equation*}
(\Delta-\lambda)(-\phi_0) = -(\lambda_1(\Sigma)-\lambda) \phi_0.
\qedhere
\end{equation*}
\end{proof}

\subsubsection{Construction of the quasimodes for cylinders.}

 The construction of the quasimodes  on the cylinder is almost completely analogous to that on the cross cap.
 We have to use the kernel $J_\lambda$ that has poles at both, $x_0$ and $x_1$ in this case.
 
Again, we denote by $\psi_{\eps,h}$ a normalized $\lambda_0(C_{\eps,h})$-eigenfunction.
Recall that
 $$
 \frac{1}{2} \int_{\partial C_{\eps,h}} \partial_\nu \psi_{\eps,h} d \mathcal{H}^1
 =
 - 2 \left( \frac{h}{\pi} \right)^{3/2} \eps^{1/2},
 $$
corresponds to integrating the normal derivative along each of the two connected components of $\partial C_{\eps,h}$.
 
Let $\rho_h \colon [0,h] \to [0,1]$ be a smooth function such that $\rho_h(0)=1, \rho_h(h)=0$ and $|\rho_h'| \leq 2/h$.
We define the quasimode $\tilde \psi_{\eps,h}$ as follows:
For $y \in C_{\eps,h}$ we take
\begin{equation*}
\begin{split}
\tilde \psi_{\eps,h}(y)
= &
\psi_{\eps,h}(y) + 
2 \left( \frac{h}{\pi} \right)^{3/2} \eps^{1/2}
\rho_h 
((\frac{1}{2\pi} \log(1/\eps^k)+e_{\lambda,0}(x_0)+H_{\lambda,1}(x_0)) 
\\
&+
2 \left( \frac{h}{\pi} \right)^{3/2} \eps^{1/2}(1-\rho_h)(\frac{1}{2\pi} \log(1/\eps^k)+e_{\lambda,1}(x_1)+H_{\lambda,0}(x_1)).
\end{split}
\end{equation*}
If $y \in \Sigma \setminus (B_{2\eps^k} (x_1) \cup B_{\eps^k} (x_0))$, we put
$$
\tilde \psi_{\eps,h}(y)
=
2 \left( \frac{h}{\pi} \right)^{3/2} \eps^{1/2} \left(\eta_\eps J_\lambda(y) + (1-\eta_\eps)   \left(\frac{1}{2\pi}\log \left( \frac{1}{|x_0-y|} \right) + e_{\lambda,0}(x_0) + H_{\lambda,1}(x_0) \right) \right).
$$
In the remaining case that $y \in B_{2 \eps^k} (x_1)$ we define
$$
\tilde \psi_{\eps,h}(y)
=
2 \left( \frac{h}{\pi} \right)^{3/2} \eps^{1/2} \left(\eta_\eps J_\lambda(y) + (1-\eta_\eps)   \left(\frac{1}{2\pi}\log \left( \frac{1}{|x_1-y|} \right) + e_{\lambda,1}(x_1) + H_{\lambda,0}(x_1) \right) \right) 
$$
 
As before, let $u_{\eps,h}$ be an eigenfunctions with eigenvalue $\lambda_{\eps,h} \in [\delta_0,\lambda_{K+1}(\Sigma)-\delta_0]$ as in \eqref{eq_assum_conc_cyl} above (with $M_{\eps,h}$ replaced by $C_{\eps,h}$).
Then we have thanks to \cref{conv} and our choice of $h_0$ at \eqref{eq_choice_h_0} that
$$
\left| \int_{\Sigma_{\eps,h}} \tilde \psi_{\eps,h} u_{\eps,h} \right| \geq c_1
$$ 
for $\eps$ sufficiently small.

Thanks to \eqref{eq_kernel_handle} the arguments from \cref{sub_sec_quasi_cross} along with some minor modifications give the following.

\begin{lemma} \label{lem_asymp_cyl}
Let $\tilde \psi_{\eps,h}$  be the quasimode defined above.
Then the eigenvalue $\lambda_{\eps,h}$ has the asymptotic expansion
\begin{equation} \label{eq_asymp_1}
\begin{split}
\lambda_{\eps,h}
&= 
\lambda  - 2 \left( \frac{h}{\pi} \right)^{3/2} \frac{\frac{\lambda}{\lambda_1(\Sigma)} (\phi_0(x_0)+\phi_0(x_1) )\int_{\Sigma \setminus B_{2\eps^k}} \phi_0 u_{\eps,h} \eps^{1/2}}{  \int_{\Sigma_{\eps,h}}  u_{\eps,h} \tilde \psi_{\eps,h}}
\\
&+ O(\eps\log(1/\eps^k))
\end{split}
\end{equation}
as $\eps \to 0$, uniformly in $h \in [h_0,h_1]$ as long as $u_{\eps,h}$ satisfies \eqref{eq_assum_conc_cyl} with $c_0>0$ fixed.
\end{lemma}

The term of order $\eps \log(1/\eps)$ looks essentially the same as for cross caps.
However, we do not need it in order to prove the second item of \cref{main_1}.

Note that under the symmetry assumption \eqref{eq_assum_key_cancel} the first term on the right hand side vanishes.
Since the second summand is of order $\eps \log(1/\eps)$ thanks to H{\"o}lder's inequality, we can locate the corresponding eigenvalue up to scale $\eps \log(1/\eps)$.
Besides  \cref{mono_quasimode_2} this is the second crucial ingredient to obtain the second part of \cref{main_1}.

\section{Proofs of main results}\label{sec5}

\subsection{Surfaces with symmetries}\label{proofmain1}

In this section we prove \cref{main_1}. 
The first part is straighforward using the convergence result for the spectrum \cref{conv} and
the quasimodes from \cref{mono_quasimode_1}.
The second part is more subtle and requires a careful choice of the height parameter $h$ adjusted to the radius $\eps$ in order to keep the branch corresponding to $\lambda_0(C_{\eps,h})$ not too much below $\lambda_1(\Sigma)$, while simultaneously having a good quasimode in \cref{mono_quasimode_2}.

\begin{proof}[Proof of \cref{main_1} (i)]
Recall that $h_0>0$ is chosen such that 
$$
\lambda_0(M_{\eps,h_0}) > \lambda_1(\Sigma).
$$
Therefore, by \cref{conv}, for $\eps$ sufficiently small, we can find some fixed small $\delta>0$
such that the first $K$
non-trivial eigenvalues of $\Sigma_{\eps,h}$ are contained in the interval $[\lambda_1(\Sigma)-\delta,\lambda_1(\Sigma)+\delta]$
and
$\lambda_{K+1}(\Sigma_{\eps,h}) \geq \lambda_1(\Sigma) + 2\delta$.
On the other hand, if we take an orthonormal basis $(\phi_0, \dots, \phi_{K-1})$ of $\lambda_1(\Sigma)$-eigenfunctions, we have for the quasimodes 
given by \eqref{eq_quasimod_surf} that
$$
\left |\int_{\Sigma_{\eps,h}} (\phi_i)_\eps (\phi_j)_\eps - \delta_{ij} \right|
\leq C \eps^{2k}.
$$
Therefore, it easily follows from \cref{lem_anne_quasimod} applied to the quasimodes from \cref{mono_quasimode_1} that there are at least $K$
 eigenvalues in $[\lambda_1(\Sigma) - \eps^{k/4} , \lambda_1(\Sigma) + \eps^{k/4}]$ for $\eps$ sufficiently small (cf.\ \cref{rem_anne_quasimod}).
Clearly, this implies that we need to have
$$
\lambda_1(\Sigma_{\eps,h_0}) \geq \lambda_1(\Sigma) - \eps^{k/4}
$$
for $\eps$ sufficiently small.
If we combine this with the area bound
$$
\area(\Sigma_{\eps,h_0})  \geq \area(\Sigma) + 2\pi h_0 \eps + O(\eps^{2k})
$$
we immediately obtain that
\begin{align*}
\lambda_1(\Sigma_{\eps,h_0}) \area(\Sigma_{\eps,h_0}) 
&\geq
\lambda_1(\Sigma) \area(\Sigma) + 2 \pi h_0 \lambda_1(\Sigma) \eps - O(\eps^{k/4})
\\
&> 
\lambda_1(\Sigma)\area(\Sigma)
\end{align*}
for $\eps$ sufficiently small since $k>4$.
\end{proof}

\begin{proof}[Proof of \cref{main_1} (ii)]
Recall once again that $h_*>0$ is chosen such that
$$
\lambda_1(\Sigma)=\lambda_0(C_{\eps,h_*}).
$$
We define $h_\eps>0$ by requiring that 
\begin{equation*}
\frac{ \pi^2}{ h_\eps^2}=  \frac{ \pi^2}{ h_*^2}+ \eps^{3/4}.
\end{equation*}
Let $k>4$ and define $\Sigma_\eps=\Sigma_{\eps,h_\eps}$.
We now take an orthonormal basis $(\phi_0,\dots, \phi_{K-1})$ of $\lambda_1(\Sigma)$-eigenfunctions with the property that $\phi_1(x_i)=\dots = \phi_{K-1}(x_i)=0$, which we can always do for dimensional reasons. 
Let us denote by $(\phi_i)_\eps$ the quasimode associated to $\phi_i$ constructed before 
\cref{mono_quasimode_2}.
We then have that
\begin{equation*}
\left| \int_{\Sigma_{\eps,h}} (\phi_i)_\eps (\phi_j)_\eps -\delta_{ij} \, \right| \leq C \eps^{2k}.
\end{equation*}
Therefore it follows from \cref{lem_anne_quasimod} and \cref{mono_quasimode_2} that there are at least $K$ eigenvalues
in $[\lambda_1(\Sigma)-C\eps^{5/4}, \lambda_1(\Sigma) + C \eps^{5/4}]$ (cf.\ \cref{rem_anne_quasimod}).
Thanks to our choice of $h_\eps$, we know from \cref{conv} that there are exactly $K+1$ eigenvalues contained in $[\lambda_1(\Sigma)-\delta,\lambda_1(\Sigma)+\delta]$ 
for some fixed small $\delta>0$ and $\eps$ sufficiently small.
At this stage, we have located $K$ of these within the interval $[\lambda_1(\Sigma)-C \eps^{5/4},\lambda_1(\Sigma)+C\eps^{5/4}]$.

We now prove the following\footnote{Alternatively, we could assume that the remaining eigenvalue is not contained in $[\lambda_1(\Sigma)-C \eps^{5/4},\lambda_1(\Sigma)+C\eps^{5/4}]$ (since we would be finished otherwise) and obtain the same conclusion directly from \cref{mono_quasimode_2} and \cref{conv}.
However, we need the more general statement here in the proof of \cref{main_2}.}

\begin{claim} \label{claim_conc}
There is $c_0>0$ such that for any $\eps$ sufficiently small there is a normalized eigenfunction $v_{\eps,h}$ with eigenvalue $\lambda_{\eps,h} \in [\delta,\lambda_0(C_{\eps,h_0})+\delta]$ with
$$
\int_{C_{\eps,h}} |v_{\eps,h}|^2 \geq c_0.
$$
\end{claim}

\begin{proof}[Proof of \cref{claim_conc}]
It follows from \cref{conv} that there are exactly $K+1$ normalized, orthogonal eigenfunctions $v_{\eps,h}^{(i)}$ with eigenvalues in $ \lambda_{\eps,h}^{(i)} \in [\delta,\lambda_0(C_{\eps,h_0})+\delta]$.
We argue by contradiction and assume
$$
\lim_{\eps \to 0} \int_{C_{\eps,h}} |v_{\eps,h}^{(i)}|^2=0
$$
for any $i=1,\dots,K+1$.
Therefore, \cref{conv} implies that the harmonic extensions $\widetilde{\left. v_{\eps,h}^{(i)} \right|_{\Sigma}}$ converge to $K+1$ orthonormal  $\lambda_1(\Sigma)$-eigenfunctions.
But the multiplicity of $\lambda_1(\Sigma)$ is only $K$, a contradiction.
\end{proof}

Invoking \cref{conv} we find from \cref{claim_conc} that
$$
\left| \int_{C_{\eps,h}} v_{\eps,h} \psi_{\eps,h}  \right| \geq c_1
$$
for $\eps$ sufficiently small and some uniform constant $c_1>0$.

We can now apply \cref{lem_asymp_cyl}
to find for the corresponding eigenvalue that
\begin{equation*}
\lambda_\eps \geq \frac{ \pi^2}{h_\eps^2} - C \eps \log(1/\eps)
\geq \lambda_1(\Sigma)  + \left| \frac{ \pi^2}{h_\eps^2} - \frac{ \pi^2}{h_*^2}\right| - C \eps \log(1/\eps)
\geq \lambda_1(\Sigma) +  \eps^{3/4}/2
\end{equation*}
for $\eps$ sufficiently small.
In particular, since we have located all the relevant eigenvalues, we can conclude that
\begin{equation*}
\lambda_1(\Sigma_\eps) \geq \lambda_1(\Sigma)-C\eps^{5/4}.
\end{equation*}
The assertion follows now exactly as in the proof of the first part.
\end{proof}

\subsection{Surfaces without symmetries}

We now give the proof of \cref{main_2} using the two quasimodes contructed in \cref{sec_quasi_handle}.

\begin{proof}[Proof of \cref{main_2}]
Let us start with some general considerations that apply to any $h \in [h_0,h_1]$.
By \cref{claim_conc} we can find a normalized eigenfunction $u_{\eps,h}$ with eigenvalue $\lambda_{\eps,h} \in [\delta, \lambda_0(M_{\eps,h_0})-\delta]$ for some small fixed $\delta>0$ such that
\begin{equation} \label{eq_mass_conc}
\int_{M_{\eps,h}} |u_{\eps,h}|^2 \geq c_0.
\end{equation}
for some $c_0>0$.
Thanks to the last part of \cref{conv}, up to multiplying $u_{\eps,h}$ by $-1$, we may therefore assume that
\begin{equation} \label{eq_sign}
\int_{M_{\eps,h}} \psi_{\eps,h} u_{\eps,h} \geq c_1
\end{equation}
for some uniform $c_1=c_1(c_0,h_0,h_1)>0$.
We now want to use the asymptotic expansions \cref{lem_asymp_1} and \cref{lem_asymp_2} both applied to the 
eigenvalue $\lambda_{\eps,h}$.

\smallskip

To simplify notation, let us define
\begin{equation} \label{eq_def_beta}
\beta_{\eps,h}:=
\frac{\int_{\Sigma_{\eps,h}}u_{\eps,h}(\tilde \psi_{\eps,h}-\chi_{\eps,h})}{\int_{\Sigma_{\eps,h}}u_{\eps,h}\chi_{\eps,h}}.
\end{equation}
Observe that the assumption \eqref{eq_sign} implies that the denominator of the fraction is bounded away from zero for $\eps$ sufficiently small.
Moreover, recalling \eqref{eq_diff_quasimod}, we have that
\begin{equation} \label{eq_comp_quasimod}
\begin{aligned}
\int_{\Sigma_{\eps,h}} & u_{\eps,h}(\tilde \psi_{\eps,h}-\chi_{\eps,h})=
\\
&=
\frac{1+e_{\eps,\lambda}}{2 \pi} \left(\frac{2 \pi}{h}\right)^{3/2} \log\left({1}/{\eps^k}\right) \eps^{1/2}
\left(
 \int_{M_{\eps,h}} u_{\eps,h}
+
 \int_{\Sigma \setminus B_{\eps^k}} \frac{\phi_{0,\eps}}{\phi_0(x_0)} u_{\eps,h}
\right)
\\
&=
\frac{1+e_{\eps,\lambda}}{2 \pi} \left(\frac{2 \pi}{h}\right)^{3/2} 
\left( 
\left(
\int_{\Sigma \setminus  B_{2 \eps^k}} \frac{\phi_0}{\phi_0(x_0)} u_{\eps,h} 
\right)
\eps^{1/2} \log(1/\eps^k)
+ 
O(\eps \log(1/\eps^k))
\right)
\end{aligned}
\end{equation}
as $\eps \to 0$ by \cref{rem_order}.
In particular, we find that
\begin{equation} \label{eq_order_beta}
\beta_{\eps,h}=O(\eps^{1/2}\log(1/\eps^k))
\end{equation}
as long as \eqref{eq_mass_conc} holds.
Note that
$$
\int_{\Sigma_{\eps,h}} u_{\eps,h} \tilde \psi_{\eps,h}
=
(1+\beta_{\eps,h})
\int_{\Sigma_{\eps,h}} u_{\eps,h} \chi_{\eps,h}.
$$
 By comparing the asymptotic expansions obtained from \cref{lem_asymp_1} and \cref{lem_asymp_2} respectively, we find that

  \begin{multline*}
  \frac{\lambda}{\lambda_1(\Sigma)} \int_{\Sigma \setminus B_{2\eps^k}} \phi_0(x_0) \phi_0 u_{\eps,h} \eps^{1/2}
   +\frac{\lambda}{2\pi} (1+e_{\eps,\lambda}) \int_{M_{\eps,h}} u_{\eps,h} \log(1/\eps^k)\eps^{1/2}  
 + O(\eps^k \log(1/\eps))
 \\
 =
 (1+\beta_{\eps,h})
   \frac{\lambda}{\lambda_1(\Sigma)} \int_{\Sigma \setminus B_{2\eps^k}} \phi_0(x_0) \phi_0 u_{\eps,h} \eps^{1/2}
  \\ 
  +
   (1+\beta_{\eps,h})
   \frac{\lambda_1(\Sigma)-\lambda}{2\pi} (1+e_{\eps,\lambda}) \int_{\Sigma \setminus B_{2\eps^k}} \frac{\phi_0}{\phi_0(x_0)}u_{\eps,h} \eps^{1/2}\log(1/\eps^k)
   \end{multline*}
  Thanks to \eqref{eq_order_beta} this implies that
   
  \begin{multline} \label{eq_leading_order}
  \frac{1+e_{\eps,\lambda}}{2\pi} \left( \int_{M_{\eps,h}} u_{\eps,h}\right) \log(1/\eps^k)\eps^{1/2}  
  -
 \frac{\beta_{\eps,h}}{\lambda_1(\Sigma)} \left(\int_{\Sigma \setminus B_{2\eps^k}} \phi_0(x_0) \phi_0 u_{\eps,h}\right) \eps^{1/2}
 =
 \\
\frac{\lambda_1(\Sigma)-\lambda}{2\pi \lambda} (1+e_{\eps,\lambda}) \int_{\Sigma \setminus B_{2 \eps^k}} \frac{\phi_0}{\phi_0(x_0)}u_{\eps,h} \eps^{1/2} \log(1/\eps^k)\\
 -
  |\lambda_1-\lambda|O(\eps \log^2(1/\eps^k)) + O(\eps^k \log(1/\eps)).
 \end{multline}

We now write
 \begin{equation*} \label{eq_mass_quanti}
 \int_{M_{\eps,h}} u_{\eps,h} \psi_{\eps,h} = n_{\eps,h}
  \end{equation*}
  and
  \begin{equation*}
  \int_{\Sigma \setminus B_{2 \eps^k}} u_{\eps,h} \phi_0 =m_{\eps,h}
  \end{equation*}
  for some $n_{\eps,h} \in [c_1,1)$ and $m_{\eps,h} \in (-1,1)$ thanks to \eqref{eq_sign}.
We use corresponding notation also for other eigenfunctions explicitly indicating the eigenfunction whenever necessary.

\smallskip
 
With this notation, using \eqref{eq_diff_quasimod} we then find from \eqref{eq_comp_quasimod} that
\begin{align*}
\frac{\beta_{\eps,h}}{\lambda_1(\Sigma)} &\int_{\Sigma \setminus B_{2\eps^k}} \phi_0(x_0)\phi_0 u_{\eps,h} \eps^{1/2} \\
&=
\frac{1}{2 \pi \lambda_1(\Sigma)} \left( \frac{2\pi}{h} \right)^{3/2}  \frac{(1+e_{\eps,\lambda})(m_{\eps,h}^2 + O(\eps^{3/2}\log(1/\eps)) }{n_{\eps,h}} \eps \log(1/\eps^k) .
\end{align*}
Moreover, we have from the last part of \cref{conv} that
\begin{equation*}
\int_{M_\eps,h} \left| u_{\eps,h} - n_{\eps,h} \psi_{\eps,h} \right|^2 \leq C \eps \log(1/\eps),
\end{equation*}
which combined with \eqref{eq_l1_norm} and H{\"o}lder's inequality implies that
\begin{equation} \label{eq_approx_l1} 
\begin{split}
 \int_{M_{\eps,h}} u_{\eps,h}
 &= 
 \int_{M_{\eps,h}} n_{\eps,h} \psi_{\eps,h} + \int_{M_{\eps,h}} (u_{\eps,h} - n_{\eps,h} \psi_{\eps,h})
 \\
 &=
 4 \left(\frac{h}{2\pi}\right)^{1/2}n_{\eps,h} \eps^{1/2}+O(\eps \log^{1/2}(1/\eps)).
 \end{split}
\end{equation}
Therefore, for $u_{\eps,h}$, \eqref{eq_leading_order} can be written as
\begin{equation} \label{eq_u_final}
\begin{split}
& \left( 
\frac{1}{2\pi \lambda_1(\Sigma)} \left( \frac{2\pi}{h} \right)^{3/2} \frac{m_{\eps,h}^2}{n_{\eps,h}} -
\frac{4}{2\pi} \left( \frac{h}{2\pi} \right)^{1/2} n_{\eps,h} 
 \right)
 \eps \log(1/\eps^k)
 \\
 &=
- \frac{\lambda_1(\Sigma)-\lambda}{2 \pi \lambda}  \frac{m_{\eps,h}}{\phi_0(x_0)}   \eps^{1/2}\log(1/\eps^k) + |\lambda - \lambda_1(\Sigma)|O(\eps \log^2(1/\eps^k)) + O(\eps^2 \log^2(1/\eps)).
  \end{split}
\end{equation}

Since by assumption $n_{\eps,h} \geq c_1$, we obtain after dividing by $n_{\eps,h}$, and considering the leading order term after some easy simplifications that 
\begin{equation} \label{eq_dis_final}
\frac{m_{\eps,h}^2}{n_{\eps,h}^2} = f_\eps^2(h) + O(\eps^{3/2}\log(1/\eps)),
\end{equation}
where $f_\eps$ is a solution to \eqref{eq_def_f}. 

Without further specifying our choice of the parameter $h$ and of the eigenfunction beyond \eqref{eq_mass_conc} this is all we can conclude.
Let us now fix some $D>0$ and consider $h \in [h_*-D\eps^{1/2},h_*+D \eps^{1/2}]$, so that 
$$
|\lambda_0(M_{\eps,h}) - \lambda_1(\Sigma)| \leq D' \eps^{1/2}
$$
for some other constant $D'>0$.
Under these assumptions it follows from \eqref{eq_u_final}
that 
\begin{equation} \label{eq_mass_u_sigma}
m_{\eps,h}^2 \geq d_0^2
\end{equation}
for some constant $d_0>0$ since if $m_{\eps,h}$ were of size $o(1)$, the second term on the left hand side of \eqref{eq_u_final} would be the only one left of size comparable to $\eps \log(1/\eps)$.
Since $m_{\eps,h}^2 + n_{\eps,h}^2 \leq \| u_{\eps,h}\|_{L^2(\Sigma_{\eps,h})}^2=1$,
this in turn implies that 
$$
n_{\eps,h}^2 \leq 1-d_0^2.
$$
Therefore, by an argument identical to \cref{claim_conc}\footnote{One could also use \cref{rem_anne_quasimod}},
there has to be another eigenfunction $v_{\eps,h}$ with $\|v_{\eps,h}\|_{L^2(M_{\eps,h})} \geq c_2$ for some $c_2>0$ and $\eps$ sufficiently small.
By \cref{conv} this implies that
\begin{equation} \label{eq_mass_v_cross_cap}
\left|\int_{M_{\eps,h}} v_{\eps,h} \psi_{\eps,h} \right| \geq c_3
\end{equation}
for some $c_3>0$ and $\eps$ sufficiently small.
As above, this implies that also
\begin{equation} \label{eq_mass_v_sigma}
\left|\int_{\Sigma \setminus B_{2 \eps^k}} v_{\eps,h} \phi_0 \right| \geq d_1.
\end{equation}
for some $d_1>0$ and $\eps$ sufficiently small.
Of course, the arguments leading to \eqref{eq_u_final} and thus to \eqref{eq_dis_final} also apply to $v_{\eps,h}$.
We now decompose the quasimode $\phi_{0,\eps}:=(\phi_0)_{\eps}$ (constructed before \cref{mono_quasimode_1}) into eigenfunctions,
\begin{equation} \label{eq_decomp}
\phi_{0,\eps} = 
(m_{\eps,h}(u_{\eps,h}) +O(\eps^{1/2}))u_{\eps,h} 
+ 
(m_{\eps,h}(v_{\eps,h})+O(\eps^{1/2}))v_{\eps,h}
+
\sum_{i=1}^{K-1} \alpha_{\eps,h}^i u_{\eps,h}^i 
+ 
r_{\eps,h},
\end{equation}
where we denote by $u_{\eps,h}^i$ those of the first $K+1$ non-trivial eigenfunctions of $\Sigma_{\eps,h}$ that do not correspond to $u_{\eps,h}$ or $v_{\eps,h}$, real numbers $\alpha_{\eps,h}^i \in [-2,2]$, and by $r_{\eps,h}$ the spectral projection of $\phi_{0,\eps}$ to $\{0\} \cup [\lambda_{K+2}(\Sigma_{\eps,h}),\infty)$.
Integrating \eqref{eq_decomp} against $\psi_{\eps,h}$ gives
\begin{equation*}
\begin{split}
m(u_{\eps,h})n(u_{\eps,h}) + m(v_{\eps,h})n(v_{\eps,h})
&=
\int_{\Sigma_{\eps,h}} \phi_{0,\eps} \psi_{\eps,h} 
- 
\sum_{i=1}^{K-1} \alpha_{\eps,h}^i \int_{\Sigma_{\eps,h}} u_{\eps,h}^i \psi_{\eps,h} 
\\
& \ \ - 
\int_{\Sigma_{\eps,h}} r_{\eps,h} \psi_{\eps,h}
+ 
O(\eps^{1/2}),
\end{split}
\end{equation*}
thanks to H{\"o}lder's inequality.
We now estimate the right hand side term by term.
The first term is easily handled by
$$
\int_{\Sigma_{\eps,h}} |\phi_{0,\eps} \psi_{\eps,h}|
=
|\phi_0(x_0)| \int_{M_{\eps,h}} |\psi_{\eps,h}|
\leq
C \eps^{1/2}.
$$
Thanks to \cref{conv} we know that $\lambda_{K+2}(\Sigma_{\eps,h}) \geq \lambda_{K+1}(\Sigma_{\eps,h}) + s_0$
uniformly in $h \in [h_0,h_1]$ for $\eps$ sufficiently small and some $s_0>0$.
Therefore, it follows from \cref{lem_anne_quasimod} and \cref{mono_quasimode_1} that
$$
\int_{\Sigma_{\eps,h}} |r_{\eps,h}|^2 \leq C \eps,
$$
which implies
$$
\left| \int_{\Sigma_{\eps,h}} r_{\eps,h} \psi_{\eps,h} \right|
\leq C \eps^{1/2}.
$$
Before we can estimate the third term we have to observe that we can apply the asymptotic expansion from \cref{lem_asymp_1} to $u_{\eps,h}$ to find that
$$
|\lambda_{\eps,h}-\lambda_1(\Sigma)| \geq s_1 \eps^{1/2}
$$
for some $s_1>0$
thanks to \eqref{eq_sign}, \eqref{eq_approx_l1}, and \eqref{eq_mass_u_sigma}.
Analogously, we find the same bound for the eigenvalue $\lambda_{\eps,h}'$ corresponding to $v_{\eps,h}$.
Thus, there are exactly $K-1$ eigenvalue contained in $[\lambda_1(\Sigma)-\eps^{k/2},\lambda_1(\Sigma)+\eps^{k/2}]$
for $\eps$ sufficiently small.

Therefore, for the third term, we find from \cref{lem_anne_quasimod} and \cref{mono_quasimode_1} that for any quasimode
$\phi_{j,\eps}$ with $j \geq 1$ there is a linear combination $w_{\eps,h,j}$ of the eigenfunctions $u_{\eps,h}^i$ 
such that
$$
\int_{\Sigma_{\eps,h,j}} |\phi_{j ,\eps} - w_{\eps,h,j} |^2 \leq C \eps^2
$$
since the quasimodes $\phi_{\eps,1},\dots,\phi_{\eps,K}$ are $L^2$-orthogonal up to an error of size $O(\eps^k)$ and the space spanned by them has dimension precisely $K-1$ it follows that also given any $u_{\eps,h}^j$, 
there is a linear combination $w_{\eps,h,j}'$of the quasimodes $\phi_{j,\eps}$ such that
$$
\int_{\Sigma_{\eps,h,j}} |u_{\eps,h}^i - w_{\eps,h,i}' |^2 \leq C \eps^2.
$$
In particular, this implies that
\begin{equation*}
\begin{split}
\int_{\Sigma_{\eps,h}} |\psi_{\eps,h} u_{\eps,h}^i| 
&\leq
\int_{\Sigma_{\eps,h}}|\psi_{\eps,h} w_{\eps,h,i}'| 
+ 
\int_{\Sigma_{\eps,h}}|\psi_{\eps,h}| |u_{\eps,h}^i - w_{\eps,h,i}' |
\\
&\leq
\left( \int_{\Sigma_{\eps,h,j}} |u_{\eps,h}^i - w_{\eps,h,i}' |^2 \right)^{1/2}
\leq
C \eps.
\end{split}
\end{equation*}
In conclusion, we find that
\begin{equation} \label{eq_comp_u_v}
|m(u_{\eps,h})n(u_{\eps,h}) + m(v_{\eps,h})n(v_{\eps,h})|
\leq 
C \eps^{1/2}.
\end{equation} 
Suppose now that $m_{\eps,h}(u_{\eps,h})/n_{\eps,h}(u_{\eps,h})$ corresponds to the negative solution of \eqref{eq_def_f}, 
 we then find that from \eqref{eq_comp_u_v} that  $m_{\eps,h}(v_{\eps,h})/n_{\eps,h}(v_{\eps,h})$ corresponds up to an error of size $O(\eps^{1/2})$ to the positive solution of \eqref{eq_def_f}\footnote{Note that if $x$ is a solution to \eqref{eq_def_f}, the second solution is given by $-1/x$.}.
 Therefore, up to reversing the roles of $u_{\eps,h}$ and $v_{\eps,h}$ we may assume that $m_{\eps,h}(u_{\eps,h})/n_{\eps,h}(u_{\eps,h})$ corresponds to the positive solution of \eqref{eq_def_f}.
 Therefore, by integrating $\phi_{0,\eps}$ against $u_{\eps,h}$ we find from \cref{mono_quasimode_1} that
 \begin{equation} \label{eq_conclusion}
 \lambda_{\eps,h} = 
 \lambda_1(\Sigma) - f_\eps(h)^{-1} \lambda_1(\Sigma) \phi_0(x_0) \eps^{1/2} + O(\eps \log(1/\eps)).
 \end{equation}
 Moreover, by the same argument, the eigenvalue $\lambda_{\eps,h}'$ corresponding to $v_{\eps,h}$ 
 satisfies $\lambda_{\eps,h}' \geq \lambda_1(\Sigma)$.
Recall that we also know that there are precisely $K-1$ eigenvalues contained in $[\lambda_1(\Sigma)-\eps^{k/4},\lambda_1(\Sigma)+\eps^{k/4}]$.
 Hence, \cref{conv} implies that
 $$
 \lambda_1(\Sigma_{\eps,h}) = \lambda_{\eps,h}
 $$
 and we can conclude thanks to \eqref{eq_conclusion}.
\end{proof}

\appendix

\section{Proof of \cref{lem_anne_quasimod}} \label{sec_anne}

\begin{proof}
We write $g=g_1+g_2$,
where
$$
g_1 = \sum_{\{l \colon \lambda_l(\Sigma_{\eps,\kappa}) < \lambda -s\}} \langle f , u_{\eps,\kappa,l} \rangle_{L^2(\Sigma_{\eps,\kappa})} u_{\eps,\kappa,l}
$$
and $g_2=g-g_1$.
Note that
$$
\int_{\Sigma_{\eps,\kappa}} \nabla g_i \cdot \nabla f = \int_{\Sigma_{\eps,\kappa}} |\nabla g_i|^2
$$
and
$$
\int_{\Sigma_{\eps,\kappa}} g_i f = \int_{\Sigma_{\eps,\kappa}} |g_i|^2.
$$
Therefore, we find from the assumption that
$$
\int_{\Sigma_{\eps,\kappa}} |\nabla g_i|^2 \leq \lambda \int_{\Sigma_{\eps,\kappa}} |g_i|^2 + \delta \|g_i\|_{W^{1,2}(\Sigma_{\eps,\kappa})},
$$
which implies that
$$
\|g_i\|_{W^{1,2}(\Sigma_{\eps,\kappa})}^2 \leq (\lambda+1) \|g_i \|_{L^2(\Sigma_{\eps,\kappa})}^2 + \delta \|g_i\|_{W^{1,2}(\Sigma_{\eps,\kappa})}.
$$
This in turn implies that
$$
\|g_i\|_{W^{1,2}(\Sigma_{\eps,\kappa})} \leq  (\lambda+1)\|g_i\|_{L^2(\Sigma_{\eps,\kappa})} + \delta.
$$
We now distinguish two cases.
Since we assume $s \leq 1$, the conclusion trivially holds if $\|g_i\|_{L^2(\Sigma_{\eps,\kappa})} \leq \delta$.
If $\|g_i\|_{L^2(\Sigma_{\eps,\kappa})} \geq \delta$, the previous computation implies that we have
$$
\|g_i\|_{W^{1,2}(\Sigma_{\eps,\kappa})} \leq  (\lambda+2)\|g_i\|_{L^2(\Sigma_{\eps,\kappa})} 
$$
Thus, testing against $g_i$, we find that
$$
s \|g_i\|_{L^2(\Sigma_{\eps,\kappa})}^2 \leq  (\lambda+2) \delta \|g_i\|_{L^2(\Sigma_{\eps,\kappa})},
$$
from which the lemma easily follows.
\end{proof}

\section{Green's functions} \label{sec_green}

\begin{lemma}
Let $(\Sigma,g)$ be a closed Riemannian surface with $\area(\Sigma,g)=1$ and $z \in \Sigma$, then there is a unique function $G(\cdot,z) \colon \Sigma \setminus \{z\} \to \IR$
such that 
\begin{itemize}
\item[(i)] $\Delta G(\cdot,z) = \delta_z - 1$ in the sense of distributions.
\item[(ii)] In conformal coordinates centered at $z$, such that $g(z) = g_{\it{eucl}}$ in these coordinates, we have that
$$
G(x,z) = \frac{1}{2\pi} \log\left(\frac{1}{|x-z|}\right) + \psi(x),
$$
where $\psi$ is a smooth function with $\psi(0)=0$.
\end{itemize}
\end{lemma}

\begin{proof}
We take conformal coordinates $(U,x)$ centered as $z$ as in the assertion.
Let $\eta \colon \Sigma \to [0,1]$ be a cut-off function that is $1$ near $z$ and has $\supp \eta \subset U$.
Consider the function $f \colon \Sigma \setminus \{z\} \to \IR$ given in $U$ by 
$$
f(x) = \eta(x) \frac{1}{2\pi} \log \left(\frac{1}{|x-z|}\right),
$$
where $|x-z|$ is the Euclidean distance in the coordinates $(U,x)$.
Since these coordinates are conformal and the Laplace operator is conformally covariant in dimension two, it is easy to see that
\begin{equation} \label{eq_green_ortho}
\Delta f
=
\delta_z + h,
\end{equation}
where 
$$
h=2 \nabla \eta \cdot \nabla \frac{1}{2\pi} \log \left(\frac{1}{|x-z|}\right) + \frac{1}{2\pi} \log \left(\frac{1}{|x-z|}\right) \Delta \eta.
$$
is a smooth function defined on all of $\Sigma$.
It follows from \eqref{eq_green_ortho} that
$$
\int_\Sigma h = -1.
$$
Therefore, since $\area(\Sigma,g)=1$, the function $h+1$ is orthogonal to the constants.
Since the constant function are exactly the kernel of $\Delta$ as $\Sigma$ is closed, we can find a smooth function $r \colon \Sigma \to \IR$ which is unique up to the addition of constants with
$$
\Delta r = h+1.
$$
Thus we have
$$
\Delta (f - r) = \delta_y-1.
$$
By adding a constant to $f-r$ we can now easily arrange to have (ii).
Uniqueness follows immediately from the maximum principle.
\end{proof}


\bibliography{mybibfile}

\end{document}